\documentclass[10pt,reqno]{amsart}
\usepackage{amsmath}
\usepackage{amssymb}
\usepackage{amsthm}

\newlength{\defbaselineskip}
\newcommand{\setlinespacing}[1]%
           {\setlength{\baselineskip}{#1 \defbaselineskip}}

\numberwithin{equation}{section}

\newtheorem{thm}{Theorem}[section]
\newtheorem{prop}[thm]{Proposition}
\newtheorem{lem}[thm]{Lemma}
\newtheorem{cor}[thm]{Corollary}
\theoremstyle{definition}
\newtheorem{defn}[thm]{Definition}
\theoremstyle{remark}
\newtheorem{rem}[thm]{Remark}
\numberwithin{equation}{section}

\begin{document}

\title[From resolvent estimates to unique continuation]
{From resolvent estimates to unique continuation for the Schr\"odinger equation}

\author{Ihyeok Seo}


\address{Department of Mathematics, Sungkyunkwan University, Suwon 440-746, Republic of Korea}
\email{ihseo@skku.edu}

\maketitle

\begin{abstract}
In this paper we develop an abstract method to handle the problem of unique continuation
for the Schr\"odinger equation $(i\partial_t+\Delta)u=V(x)u$.
In general the problem is to find a class of potentials $V$ which allows the unique continuation.
The key point of our work is to make a direct link between the problem and
the weighted $L^2$ resolvent estimates
$\|(-\Delta-z)^{-1}f\|_{L^2(|V|)}\leq C\|f\|_{L^2(|V|^{-1})}$.
We carry out it in an abstract way, and thereby we do not need to deal with each of the potential classes.
To do so, we will make use of limiting absorption principle and Kato $H$-smoothing theorem in spectral theory,
and employ some tools from harmonic analysis.
Once the resolvent estimate is set up for a potential class, from our abstract theory
the unique continuation would follow from the same potential class.
Also, it turns out that there can be no dented surface on the boundary
of the maximal open zero set of the solution $u$.
In this regard, another main issue for us is to know which class of potentials allows the resolvent estimate.
We establish such a new class which contains previously known ones,
and will also apply it to the problem of well-posedness for the equation.
\end{abstract}


\section{Introduction and main results}

It is well known that an analytic function has a property of unique continuation which says that
it cannot vanish in any non-empty open subset of its domain of definition without being identically zero.
The property results from expanding the function in power series
because it would vanish with all its derivatives at some point in the subset.
At this point, we can relate the property to a solution of
the Cauchy-Riemann equation $\overline{\partial}u=0$ in $\mathbb{R}^2$,
since it should be complex analytic.
Similarly, the same property holds for solutions to the Laplace equation $\Delta u=0$ in $\mathbb{R}^n$, $n\geq2$,
since they are harmonic functions that are still real analytic.
Now it can be asked whether the property is shared by other partial differential equations
whose solutions are not necessarily analytic, or even smooth.
It would be an interesting problem to prove the property for such equation.

This paper is mainly concerned with the problem for solutions of the Schr\"odinger equation
\begin{equation}\label{sch}
i\partial_t\Psi(x,t)=(-\Delta+V(x))\Psi(x,t)
\end{equation}
which describes how the wave function $\Psi$
of a non-relativistic quantum mechanical system with a potential $V$ evolves over time.
In principle, the unique continuation can be viewed as one of the non-localization properties
of the wave function which are a fruitful issue in certain interpretations of quantum mechanics.

Let us first put \eqref{sch} in a reasonably general setting which is a differential inequality of the form
\begin{equation}\label{schineq}
|(i\partial_t+\Delta)u(x,t)|\leq|V(x)u(x,t)|,
\end{equation}
where $u:\mathbb{R}^{n+1}\rightarrow\mathbb{C}$ is a solution that is a function satisfying \eqref{schineq}
and $V:\mathbb{R}^n\rightarrow\mathbb{C}$ is a potential.
From now on, we call \eqref{schineq} \textit{Schr\"odinger inequality} for convenience.
Let us also make precise what we will mean by unique continuation property:

\begin{defn}
Given a partial differential equation or an inequality in $\mathbb{R}^n$,
we say that it has the unique continuation property from a non-empty open subset $\Omega\subset\mathbb{R}^n$
if its solution cannot vanish in $\Omega$ without being identically zero.
\end{defn}

The general question to ask is for which class of potentials does the unique continuation hold.
It has been studied for decades with a half-space $\Omega$ in $\mathbb{R}^{n+1}$.
Apart from the case \eqref{schineq}, the case of time-dependent potentials $V(x,t)$
was mainly studied in \cite{KS,S2,LS,S}.
The key ingredients in these works are so-called Carleman estimates\footnote{The Carleman method, which derives unique continuation from Carleman estimate, originated from the pioneering work of Carleman \cite{Ca} for elliptic equations.}
for the operator $i\partial_t+\Delta$.
The first one due to Kenig and Sogge \cite{KS} is the following estimate
\begin{equation}\label{KS-Carl}
\big\|e^{\beta\langle(x,t),\nu\rangle}u\big\|_{L_{t,x}^{\frac{2(n+2)}{n}}(\mathbb{R}^{n+1})}\leq
C\big\|e^{\beta\langle(x,t),\nu\rangle}(i\partial_t+\Delta)u\big\|_{L_{t,x}^{\frac{2(n+2)}{n+4}}(\mathbb{R}^{n+1})},
\end{equation}
where $\langle\text{ },\text{ }\rangle$ denotes the usual inner product on $\mathbb{R}^{n+1}$,
and the constant $C$ is independent of $\beta\in\mathbb{R}$ and $\nu\in\mathbb{R}^{n+1}$.
Note that the special case where $\beta=0$ becomes equivalent to
an estimate of Strichartz \cite{Str} for the inhomogeneous Schr\"odinger equation.
In this regard, the later developments \cite{IK,S2,LS} have been made to extend \eqref{KS-Carl}
to mixed norms $L_t^qL_x^r$ ($L_t^q(\mathbb{R};L_x^r(\mathbb{R}^n)$)
for which the inhomogeneous Strichartz estimate is known to hold (as in \cite{KT,Fo,V}).
This made it possible to obtain the unique continuation for potentials $V\in L_t^pL_x^s$
where $(p,s)$ lies in the scaling-critical range $2/p+n/s=2$ (see \cite{S2,LS}).
In a similar manner, these works were further extended in \cite{S} to Wiener amalgam norms
which may be of interest since they are not allowed to possess a scaling invariance (\cite{CN}).

Unfortunately the above-mentioned Carleman estimates
are no longer available in the time-independent case \eqref{schineq}.
The aim of this paper is to develop an abstract method to handle the unique continuation problem
for \eqref{schineq} and to exhibit a few useful applications of the method.
One of our main contributions is to convert the problem
to that of obtaining the following resolvent estimate on weighted $L^2$ spaces:
\begin{equation}\label{resol}
\big\|(-\Delta-z)^{-1}f\big\|_{L^2(|V|)}\leq C(V)\|f\|_{L^2(|V|^{-1})},
\end{equation}
where $z\in\mathbb{C}\setminus\mathbb{R}$, and $C(V)$ is a suitable constant depending on the potential $V(x)$
but it should be independent of $z$.
At this point, it should be noted that our method focuses more on the goal of obtaining unique continuation
directly from resolvent estimates rather than Carleman estimates.

We shall carry out the project in a unified manner,
and thereby we do not need to deal with each of the potentials which allow \eqref{resol}.
Let us first make the following definition in which we denote $A\sim B$ to mean $CA\leq B\leq CA$ with unspecified constants $C>0$.

\begin{defn}
We say that a potential class $\mathcal{R}$ is of resolvent type if
there exists a suitable function $[\,\cdot\,]_{\mathcal{R}}$ depending only on $\mathcal{R}$ itself 
such that the resolvent estimate \eqref{resol} holds with bound $[V]_{\mathcal{R}}$.
\end{defn}

\begin{rem}
Note that if $|V_1|\leq|V_2|$ and \eqref{resol} holds for $V_2$,
then it also holds for $V_1$ with $C(V_1)=C(V_2)$.
In this case, we take $[V_1]_{\mathcal{R}}$ to be less than or equal to $[V_2]_{\mathcal{R}}$.
\end{rem}

Here and thereafter, we always use the letter $\mathcal{R}$ to mean a potential class of resolvent type.
Once a suitable class $\mathcal{R}$ is set up, from our abstract theory
the unique continuation for the Schr\"odinger inequality would follow from $V\in\mathcal{R}$.
At this point, the main issue for us is to know which class of potentials can be of resolvent type.
As we will see later, the scaling-critical Lebesgue class $L^{n/2}$, with $[\,\cdot\,]_{L^{n/2}}=\|\cdot\|_{L^{n/2}}$,
is one of such classes.
In fact we establish a new and wider class which contains $L^{n/2,\infty}$ and
even the Fefferman-Phong class $\mathcal{L}^{2,p}$ (see \eqref{feff}) for $p>(n-1)/2$, $n\geq3$.
In particular, when $n=3$ it also contains the global Kato and Rollnik classes (see \eqref{K}, \eqref{R}).
These will be discussed in detail in the next section as one of the cores of our work.

Before stating our results on unique continuation,
we need to set up more notation in order to be precise.
A weight\footnote{It is a locally integrable function
which is allowed to be zero or infinite only on a set of Lebesgue measure zero,
so $w^{-1}$ is also a weight if it is locally integrable.}
$w:\mathbb{R}^{n}\rightarrow[0,\infty]$ is said to be of Muckenhoupt $A_p(\mathbb{R}^{n})$ class, $1<p<\infty$, if
there is a constant $C_{A_p}$ such that
\begin{equation}\label{ap}
\sup_{Q\text{ cubes in }\mathbb{R}^{n}}
\bigg(\frac1{|Q|}\int_Qw(x)dx\bigg)\bigg(\frac1{|Q|}\int_Qw(x)^{-\frac1{p-1}}dx\bigg)^{p-1}<C_{A_p}.
\end{equation}
Note that $w\in A_2\,\,\Leftrightarrow\,\, w^{-1}\in A_2$.
Given $v\in\mathbb{R}^{n}$, one may write for $x\in\mathbb{R}^{n}$, $x=sv+\widetilde{x}$,
where $s\in\mathbb{R}$ and $\widetilde{x}$ is in some hyperplane $\mathcal{P}$
whose normal vector is $v$.
We shall denote by $w\in A_p(v)$ to mean that $w$ is in the $A_p$ class
in one-dimensional direction of the vector $v$
if the function $w_{\widetilde{x}}(s):=w(x)$ is in $A_p(\mathbb{R})$ with $C_{A_p}$
uniformly in almost every $\widetilde{x}\in\mathcal{P}$.
At first glance this notion could be more or less complicated.
But, by translation and rotation it can be reduced to the case
where $v=(0,...,0,1)\in\mathbb{R}^{n}$ and $\mathcal{P}=\mathbb{R}^{n-1}$.
In this case $w\in A_p(v)$ means that $w(x_1,...,x_{n-1},\cdot)\in A_p(\mathbb{R})$
in the $x_{n}$ variable uniformly in almost every $\widetilde{x}=(x_1,...,x_{n-1})\in\mathbb{R}^{n-1}$.
We point out that this one-dimensional $A_p$ condition is trivially satisfied
if $w$ is in a more restrictive $A_p(\mathbb{R}^{n})$ class defined over arbitrary rectangles
instead of cubes (see Lemma 2.2 in \cite{Ku}).
For $\nu=(\nu_1,...,\nu_n,\nu_{n+1})\in\mathbb{R}^{n+1}$,
we denote by $\nu^\prime$ the vector $(\nu_1,...,\nu_n)$ in $\mathbb{R}^n$.

Let us now state our unique continuation theorems for the Schr\"odinger inequality
\begin{equation}\label{schineq1}
|(i\partial_t+\Delta)u(x,t)|\leq|V(x)u(x,t)|
\end{equation}
where $V\in\mathcal{R}$.
There are two types of the theorems, global ones and local ones.
First, the following global theorem says that if the solution of \eqref{schineq1} is supported on one side of a hyperplane in $\mathbb{R}^{n+1}$,
then it must vanish on all of $\mathbb{R}^{n+1}$.
In other words, the unique continuation arises globally from a half-space in $\mathbb{R}^{n+1}$.

\begin{thm}\label{thm2}
Let $n\geq2$.
Assume that $u\in \mathcal{H}_{t}^{1}(\mathcal{L}_2)\cap \mathcal{H}_{x}^{2}(\mathcal{L}_2)$ is a solution of \eqref{schineq1} with $V\in\mathcal{R}$ and vanishes in a half space with a unit normal vector $\nu\in\mathbb{R}^{n+1}$.
Then it must be identically zero if $|V|\in A_2(\nu^\prime)$
and $[V]_{\mathcal{R}}<\varepsilon$ for a sufficiently small $\varepsilon>0$.
\end{thm}

Let us give more details about the assumptions in the theorem.
First, $\mathcal{L}_2=L^2\cap L^2(|V|^{-1})$ is the solution space for which we have unique continuation,
and $\mathcal{H}_{t}^{1}(\mathcal{L}_2)$ denotes the space of functions whose derivatives up to order $1$,
with respect to the time variable $t$, belong to $\mathcal{L}_2$.
Similarly for $\mathcal{H}_{x}^{2}(\mathcal{L}_2)$.
It should be noted that the solution space is dense in $L^2$.
In fact, consider $D_n=\{x\in\mathbb{R}^n:|V|^{-1/2}\leq n\}$.
Then, for $f\in L^2$ the function $\chi_{D_n}f$ is contained in $\mathcal{L}_2$,
and $\chi_{D_n}f\rightarrow f$ as $n\rightarrow\infty$.
By the Lebesgue dominated convergence theorem, it follows now that $\chi_{D_n}f\rightarrow f$ in $L^2$.
So, $\mathcal{L}_2$ is dense in $L^2$.

Next, the smallness assumption such as $[V]_{\mathcal{R}}<\varepsilon$
is quite standard in the study of unique continuation.
In some cases of $\mathcal{R}$ where the assumption $|V|\in A_2(\nu^\prime)$
is superfluous, it can be replaced by a more local one
\begin{equation*}
\sup_{a\in\mathbb{R}}
\lim_{\delta\rightarrow0}[\chi_{S_{a,a+\delta}(\nu^\prime)}V]_{\mathcal{R}}<\varepsilon\quad(\text{if }\nu=(\nu^\prime,0))
\end{equation*}
which is trivially satisfied for the case $\mathcal{R}=L^{n/2}$.
Here $\chi_E$ denotes the characteristic function of a set $E$ in $\mathbb{R}^n$,
and for $a\in\mathbb{R}$ and a unit vector $v\in\mathbb{R}^n$,
$S_{a,a+\delta}(v)$ denotes a ``strip" in $\mathbb{R}^n$ with width $\delta>0$ given by
\begin{equation*}
S_{a,a+\delta}(v):=\{x\in\mathbb{R}^n:a<\langle x,v\rangle\leq a+\delta\}.
\end{equation*}
Also, the solution space can be extended to the whole space $L^2$ in such cases.
See Section \ref{sec6} for details.

\smallskip

Now we turn to another type of unique continuation which is more local in nature.
First we point out that there exists a smooth potential $V$
such that $(i\partial_t+\Delta)u=V(x,t)u$,
$0\in\text{supp}\,u$, and $u=0$ on $\{(x,t)\in\mathbb{R}^{n+1}:x_1<0\}$ in a neighborhood of the origin.
This result is from a particular case of Th\'{e}or\`{e}me 1.6 in \cite{LZ} due to Lascar and Zuily.
(See also \cite{Z}, p.\,127, Theorem 2.10.)
In this case the solution $u$ cannot vanish near the origin
across the hyperplane $\{(x,t)\in\mathbb{R}^{n+1}:x_1=0\}$ because $0\in\text{supp}\,u$.
This shows that the Schr\"odinger equation does not have, as a rule, a property of unique continuation
locally across a hyperplane in $\mathbb{R}^{n+1}$.
But, our result below says that the unique continuation
can arise locally across a hypersurface on a sphere in $\mathbb{R}^{n+1}$
into an interior region of the sphere.

\begin{thm}\label{thm6}
Let $n\geq2$ and let $S_r^n$ be a sphere in $\mathbb{R}^{n+1}$ with radius $r$.
Assume that $u\in\mathcal{H}_{t}^{1}(\mathcal{L}_2)\cap \mathcal{H}_{x}^{2}(\mathcal{L}_2)$
is a solution of \eqref{schineq1} with $V\in\mathcal{R}$
and vanishes on an exterior neighborhood of $S_r^n$ in a neighborhood of a point $p\in S_r^n$.
Let $\nu$ be the unit outward normal vector of $S_r^n$ at $p$.
Then it follows that $u\equiv0$ in a neighborhood of $p$,
if $|V|\in A_2(\nu^\prime)$ and $[V]_{\mathcal{R}}<\varepsilon$ for a sufficiently small $\varepsilon>0$.
\end{thm}

In the same cases as above,
the solution space $\mathcal{L}_2$ can be extended to the whole space $L^2$,
and the smallness assumption can be replaced by a more local one
$$\lim_{\delta\rightarrow0}[\chi_{B_\delta(p^\prime)}V]_{\mathcal{R}}<\varepsilon,$$
where $B_\delta(p)$ denotes a ball centered at $p$ with radius $\delta$.

More interestingly, the above local theorem gives us a possibility of seeing the geometric shape of the zero set
of the solutions.
Roughly speaking, there can be no dented surface on the boundary of the maximal open zero set.
Let us first introduce the following definition.

\begin{defn}
Let $n\geq2$. Then we say that a non-empty open set $\Omega\varsubsetneq\mathbb{R}^{n}$ has a dent
at a point $p$ in the boundary $\partial\Omega$,
if there is a sphere $S_r^{n-1}$ in $\mathbb{R}^n$ such that $p\in S_r^{n-1}$ and
an exterior neighborhood of $S_r^{n-1}$ in a neighborhood of $p$ is contained in $\Omega$.
\end{defn}

The following corollary is now deduced from the above local unique continuation.

\begin{cor}
Let $\mathcal{M}$ be the maximal open set in $\mathbb{R}^{n+1}$ on which
the solution $u$ of \eqref{schineq1} with $V\in\mathcal{R}$ vanishes.
Then the boundary $\partial\mathcal{M}$ cannot have a dent
if the same assumptions as in Theorem \ref{thm6} hold.
\end{cor}

\begin{proof}
Indeed, if there is a dent at $p\in\partial\mathcal{M}$,
then it is clear that $u$ vanishes on an exterior neighborhood of a sphere $S_r^{n}$ in a neighborhood of $p$.
From the above local unique continuation, $u$ must vanish in a neighborhood of $p$,
and so $p\in\mathcal{M}$. But this contradicts the maximality of $M$.
\end{proof}

Finally, let us sketch the organization of the paper.
In Section \ref{sec2} we establish a new function class of $V$ which allows the resolvent estimate \eqref{resol}.
Then, new results on the unique continuation are immediately obtained
from the abstract global and local theorems that are proved in Sections \ref{sec6} and \ref{sec7}, respectively.
The key ingredient in the proof is the following abstract Carleman estimate
\begin{equation}\label{Carle}
\big\|e^{\beta\langle(x,t),\nu\rangle}u\big\|_{L_{t,x}^2(|V|)}\leq
C[V]\big\|e^{\beta\langle(x,t),\nu\rangle}(i\partial_t+\Delta)u\big\|_{L_{t,x}^2(|V|^{-1})},
\end{equation}
where $C$ is a constant independent of $\beta\in\mathbb{R}$ and $\nu\in\mathbb{R}^{n+1}$,
and $[V]$ denotes the least constant $C(V)$ for which the resolvent estimate \eqref{resol} holds.
In Section \ref{sec5} we will show \eqref{Carle} which means that a potential for the resolvent estimate
can be made to work for the Carleman estimate.
In fact our basic strategy is to derive \eqref{Carle} only from \eqref{resol} in an abstract way.
This is done in several steps and based on the following weighted $L^2$ estimate
\begin{equation}\label{propa}
\bigg\|\int_{-\infty}^\infty e^{i(t-s)\Delta}F(\cdot,s)ds\bigg\|_{L_{t,x}^2(|V|)}\leq CC(V)\|F\|_{L_{t,x}^2(|V|^{-1})},
\end{equation}
where $e^{it\Delta}$ is the free Schr\"odinger propagator (see \eqref{du}),
and $C(V)$ is the constant in the resolvent estimate \eqref{resol}.
In Section \ref{sec4} we derive \eqref{propa} only from \eqref{resol} in two ways.
The first is a concrete one using a Fourier restriction estimate
which results from a limiting absorption principle obtained in Section \ref{sec3}.
On the other hand, the second is a more direct one appealing to Kato $H$-smoothing theory.
Our new resolvent estimates may have further applications for other related problems.
In fact we will apply them in Section \ref{sec8}
to the problem of well-posedness for the Schr\"odinger equation.
The method in this paper can be also applied to the case of time-dependent potentials $V(x,t)$ such that
\begin{equation*}
\sup_{t\in\mathbb{R}}|V(x,t)|\leq W(x)\in\mathcal{R}.
\end{equation*}
It takes up the final section, Section \ref{sec9}.

\smallskip

Throughout this paper, the letter $C$ stands for positive constants possibly different at each occurrence.
We also denote by $\widehat{f}$ and $\mathcal{F}^{-1}(f)$ the Fourier and the inverse Fourier transforms of $f$, respectively.


\section{Resolvent estimates}\label{sec2}
The aim of this section is twofold: firstly to look at which class of potentials can be of resolvent type,
and secondly to establish such a new class, extending and generalizing those in \cite{CS} and \cite{BBRV}, respectively.

Let us consider the resolvent $(-\Delta-z)^{-1}$ for $z\in\mathbb{C}\setminus\mathbb{R}$.
We shall use the standard notation $R_0(z)=(-\Delta-z)^{-1}$ for convenience.
(See the next section for details.)

We start with the scaling-critical Lebesgue class $L^{n/2}$.
Indeed, by the scaling $(x,t)\rightarrow(\lambda x,\lambda^2t)$,
$u_\lambda(x,t)=u(\lambda x,\lambda^2t)$ takes the equation
$i\partial_tu+\Delta u=V(x)u$ into
$i\partial_tu_\lambda+\Delta u_\lambda=V_\lambda(x)u_\lambda$,
where $V_\lambda(x)=\lambda^2V(\lambda x)$.
Thus, $\|V_\lambda\|_{L^p}=\lambda^{2-n/p}\|V\|_{L^p}$
and the $L^p$ norm of $V_\lambda$ is independent of $\lambda$ precisely when $p=n/2$.
In what follows, it will be convenient to keep in mind that a potential class
is said to be scaling invariant if it is invariant
under the scaling $V_\lambda(x)=\lambda^2V(\lambda x)$
forced by the Schr\"odinger equation onto the potential $V$ as above.
As we will see below, the class $L^{n/2}$ is of resolvent type with $[\,\cdot\,]_{L^{n/2}}=\|\cdot\|_{L^{n/2}}$,
but it is too small to contain the inverse square potential $V(x)=a/|x|^2$ ($a>0$)
which allows the resolvent estimate,
\begin{equation*}
\|R_0(z)f\|_{L^2(a/|x|^2)}\leq Ca\|f\|_{L^2(|x|^2/a)},
\end{equation*}
due to Kato and Yajima \cite{KY}.
This potential has attracted considerable interest from mathematical physics.
This is because the Schr\"odinger operator $-\Delta+a/|x|^2$
is physically related to the Hamiltonian of a spin-zero quantum particle in a Coulomb field (\cite{Cas})
and behaves very differently depending on the value of the constant $a$ ({\it cf. \cite{RS2,RS3}}).

Now we consider a wider class of resolvent type
where we can consider singularities of the type $a/|x|^2$.
Let $\mathcal{L}^{\alpha,p}$ denote the Morrey-Campatano class which is defined for $\alpha>0$ and $1\leq p\leq n/\alpha$ by
\begin{equation}\label{feff}
V\in\mathcal{L}^{\alpha,p}\quad\Leftrightarrow\quad
\|V\|_{\mathcal{L}^{\alpha,p}}
:=\sup_{Q\text{ cubes in }\mathbb{R}^{n}}|Q|^{\alpha/n}\bigg(\frac1{|Q|}\int_{Q}|V(y)|^pdy\bigg)^{1/p}<\infty.
\end{equation}
Then the case $\alpha=2$ is of special interest for us,
since $\mathcal{L}^{2,n/2}=L^{n/2}$ and
$a/|x|^2\in L^{n/2,\infty}\subset\mathcal{L}^{2,p}$ if $p<n/2$.
Also, $\mathcal{L}^{2,p}$ is the only possible scaling-invariant Morrey-Campatano class
because $\|V(\lambda x)\|_{\mathcal{L}^{\alpha,p}}=\lambda^{-\alpha}\|V\|_{\mathcal{L}^{\alpha,p}}$.
This special class is sometimes called the Fefferman-Phong class since it was introduced
by C. Fefferman and D. H. Phong regarding spectral properties of the Schr\"odinger operator $-\Delta+V(x)$.
Thanks to a result of Chanillo and Sawyer \cite{CS},
the Fefferman-Phong class is of resolvent type
with $[\,\cdot\,]_{\mathcal{L}^{2,p}}=\|\cdot\|_{\mathcal{L}^{2,p}}$ in some range of $p$.
Precisely, for $p>(n-1)/2$, $n\geq3$,
\begin{equation}\label{fp}
\|R_0(z)f\|_{L^2(|V|)}\leq C\|V\|_{\mathcal{L}^{2,p}}\|f\|_{L^2(|V|^{-1})}.
\end{equation}

The aim here is to extend \eqref{fp} to a new and wider class of potentials.
First we need to introduce some notation.
We say that $V$ is in the Kerman-Sawyer class $\mathcal{KS}_\alpha$ for $0<\alpha<n$ if
\begin{equation}\label{ksc}
\|V\|_{\mathcal{KS}_\alpha}:=
\sup_{Q}\bigg(\int_Q|V(x)|dx\bigg)^{-1}\int_Q\int_Q\frac{|V(x)V(y)|}{|x-y|^{n-\alpha}}dxdy<\infty.
\end{equation}
Here the sup is taken over all dyadic cubes $Q$ in $\mathbb{R}^n$.
Our initial motivation for \eqref{ksc} stemmed from finding all the possible potentials $V(x)$
which allow the so-called Fefferman-Phong inequality
\begin{equation}\label{ineq}
\int_{\mathbb{R}^n}|u|^2|V|dx\leq C_V\int_{\mathbb{R}^n}|\nabla u|^2dx,
\end{equation}
where $C_V$ is a suitable constant depending on $V$.
As is well known from \cite{F}, \eqref{ineq} holds for $V\in\mathcal{L}^{2,p}$, $1<p\leq n/2$, with $C_V\sim\|V\|_{\mathcal{L}^{2,p}}$,
but it is not valid for $p=1$ as remarked in \cite{D}.
As a result of Kerman and Sawyer \cite{KeS} (see \eqref{ee} below),
the least constant $C_V$ for which \eqref{ineq} holds may be taken
to be a constant multiple of the norm $\|V\|_{\mathcal{KS}_2}$,
and so $\mathcal{L}^{2,p}\subset\mathcal{KS}_2$ if $p\neq1$.
In general, $\mathcal{L}^{\alpha,p}\subset\mathcal{KS}_\alpha$ if $p\neq1$
(see \cite{BBRV}, Subsection 2.2).
Now we define a new function class which contains $\mathcal{L}^{2,p}$ for all $p>(n-1)/2$, $n\geq3$.

\begin{defn}\label{def2.1}
Let $n\geq3$.
We say that $V$ is in the class $\mathcal{S}_n$ if $V^{\frac{n-1}2}\in\mathcal{KS}_{n-1}$,
and define a quantity $[\,\cdot\,]_{\mathcal{S}_n}$ on $\mathcal{S}_n$ by
\begin{equation*}
[V]_{\mathcal{S}_n}:=\big\|V^{\frac{n-1}2}\big\|_{\mathcal{KS}_{n-1}}^{\frac{2}{n-1}}.
\end{equation*}
\end{defn}

First, note that $\mathcal{S}_n$ is just the same as $\mathcal{KS}_2$ when $n=3$.
In this case $\mathcal{S}_n$ is also closely related to the global Kato and Rollnik classes
which are defined by
\begin{equation}\label{K}
V\in\mathcal{K}\quad\Leftrightarrow\quad
\|V\|_{\mathcal{K}}:=
\sup_{x\in\mathbb{R}^3}\int_{\mathbb{R}^3}\frac{|V(y)|}{|x-y|}dy<\infty
\end{equation}
and
\begin{equation}\label{R}
V\in\mathcal{R}\quad\Leftrightarrow\quad
\|V\|_{\mathcal{R}}:=
\int_{\mathbb{R}^3}\int_{\mathbb{R}^3}\frac{|V(x)V(y)|}{|x-y|^2}dxdy<\infty,
\end{equation}
respectively.
These are fundamental ones in spectral and scattering theory ({\it cf. \cite{K,Si}}),
and their usefulness for dispersive properties of the Schr\"odinger equation
was revealed in the recent work \cite{RoS} of Rodnianski and Schlag.
It is an elementary matter to check that
$\mathcal{K}\subset\mathcal{S}_3$ and $\mathcal{R}\subset\mathcal{S}_3$.
Also it is easy to see that $\|V^m\|_{\mathcal{L}^{\alpha,p}}=\|V\|_{\mathcal{L}^{\alpha/m,mp}}^m$ for $m>0$.
If this holds still for $\mathcal{KS}_{\alpha}$,
then $\mathcal{S}_n$ would become equivalent to $\mathcal{KS}_2$ for all dimensions.
But such property does not carry over to $\mathcal{KS}_{\alpha}$
even though $\mathcal{L}^{\alpha,p}\subset\mathcal{KS}_{\alpha}$.
Next, we point out that the class $\mathcal{S}_n$ is wider than $\mathcal{L}^{2,p}$ for all $p>(n-1)/2$.
Indeed, if $V\in\mathcal{L}^{2,p}$,
then $V^{\frac{n-1}{2}}\in\mathcal{L}^{n-1,2p/(n-1)}\subset\mathcal{KS}_{n-1}$
when $2p/(n-1)>1$, and
$$[V]_{\mathcal{S}_n}=\big\|V^{\frac{n-1}2}\big\|_{\mathcal{KS}_{n-1}}^{\frac2{n-1}}
\leq\big\|V^{\frac{n-1}2}\big\|_{\mathcal{L}^{n-1,2p/(n-1)}}^{\frac2{n-1}}=\|V\|_{\mathcal{L}^{2,p}}.$$
Finally, it should be noted that $\mathcal{S}_n$ is scaling invariant with respect to $[\,\cdot\,]_{\mathcal{S}_n}$:
$$[\lambda^2V(\lambda x)]_{\mathcal{S}_n}=\lambda^2\big\|V(\lambda x)^{\frac{n-1}2}\big\|_{\mathcal{KS}_{n-1}}^{\frac{2}{n-1}}
=\lambda^2\big(\lambda^{-(n-1)}\|V^{\frac{n-1}2}\|_{\mathcal{KS}_{n-1}}\big)^{\frac{2}{n-1}}
=[V]_{\mathcal{S}_n}.$$

Now we are ready to state the following result extending \eqref{fp} to the class $\mathcal{S}_n$.

\begin{thm}\label{thm2.3}
Let $n\geq3$. Then the class $\mathcal{S}_n$ is of resolvent type.
Namely, for $V\in\mathcal{S}_n$
\begin{equation}\label{s}
\|R_0(z)f\|_{L^2(|V|)}\leq C[V]_{\mathcal{S}_n}\|f\|_{L^2(|V|^{-1})}
\end{equation}
with a constant $C$ independent of $z\in\mathbb{C}\setminus\mathbb{R}$.
\end{thm}

\begin{rem}
It is worth comparing with Theorem 2.2 in \cite{BBRV}, which proves \eqref{s} only for $n=3$
with a different approach that does not work for higher dimensions $n\geq4$.
\end{rem}

\begin{proof}[Proof of Theorem \ref{thm2.3}]
By scaling we may first assume that $|z|=1$.
Indeed, note that
$$R_0(z)f(x)=(-\Delta-z)^{-1}f(x)=|z|^{-1}(-\Delta-z/|z|)^{-1}[f(|z|^{-1/2}\cdot)](|z|^{1/2}x).$$
So, if \eqref{s} holds for $|z|=1$, then for $z\in\mathbb{C}\setminus\mathbb{R}$
\begin{align*}
\|R_0(z)f\|_{L^2(|V|)}^2&=|z|^{-2}|z|^{-n/2}\|R_0(z/|z|)[f(|z|^{-1/2}\cdot)]\|_{L^2(|V(|z|^{-1/2}\cdot)|)}^2\\
&\leq C|z|^{-2}|z|^{-n/2}[V(|z|^{-1/2}\cdot)]_{\mathcal{S}_n}^2\|f(|z|^{-1/2}\cdot)\|_{L^2(|V(|z|^{-1/2}\cdot)|^{-1})}^2\\
&\leq C[V]_{\mathcal{S}_n}^2\|f\|_{L^2(|V|^{-1})}^2.
\end{align*}

Now we rewrite \eqref{s} in the equivalent form
\begin{equation*}
\big\||V|^{1/2}R_0(z)(|V|^{1/2}f)\big\|_{L^2}\leq C[V]_{\mathcal{S}_n}\|f\|_{L^2}
\end{equation*}
and will show this using Stein's complex interpolation ({\it cf. \cite{SW}}),
as in \cite{KRS,CS}, on an analytic family of operators $T_\lambda$ defined for $\lambda\in\mathbb{C}$ by
$$T_\lambda=|V|^{\lambda/2}(-\Delta-z)^{-\lambda}|V|^{\lambda/2}$$
with the principal branch.
First, from Plancherel's theorem we have the trivial estimate for $Re(\lambda)=0$:
\begin{equation}\label{dld}
\|T_\lambda f\|_{L^2}\leq e^{\pi|Im(\lambda)|}\|f\|_{L^2}.
\end{equation}
In fact, since $|V|^{\lambda/2}=1$ for $Re(\lambda)=0$, from Plancherel's theorem we see that
\begin{align*}
\|T_\lambda f\|_{L^2}&=\bigg\|\frac1{(|\xi|^2-z)^\lambda}\widehat{f}(\xi)\bigg\|_{L^2}\\
&\leq\sup_{\xi\in\mathbb{R}^n}\bigg|\frac1{(|\xi|^2-z)^\lambda}\bigg|\|f\|_{L^2}\\
&\leq\sup_{\xi\in\mathbb{R}^n}e^{Im(\lambda)\arg(|\xi|^2-z)}\|f\|_{L^2}\\
&\leq e^{\pi|Im(\lambda)|}\|f\|_{L^2}.
\end{align*}
On the other hand, we will get for $Re(\lambda)=(n-1)/2$
\begin{equation}\label{kern}
\|T_\lambda f\|_{L^2}\leq Ce^{\frac\pi2|Im(\lambda)|}\|V^{\frac{n-1}{2}}\|_{\mathcal{KS}_{n-1}}\|f\|_{L^2}.
\end{equation}
Then, Stein's complex interpolation between \eqref{dld} and \eqref{kern} would give
\begin{equation*}
\|T_1f\|_{L^2}\leq C\big\|V^{\frac{n-1}{2}}\big\|_{\mathcal{KS}_{n-1}}^{\frac{2}{n-1}}\|f\|_{L^2}
\end{equation*}
as desired.

It remains to show \eqref{kern}.
For this, we will use the following known integral kernel $K_{\lambda}$ of $(-\Delta-z)^{-\lambda}$
({\it cf. \cite{GS,KRS}}):
\begin{equation*}
K_{\lambda}(x)=\frac{e^{\lambda^2}2^{-\lambda+1}}{(2\pi)^{n/2}\Gamma(\lambda)\Gamma(n/2-\lambda)}
\bigg(\frac{z}{|x|^2}\bigg)^{\frac12(\frac n2-\lambda)}B_{n/2-\lambda}(\sqrt{z|x|^2}),
\end{equation*}
where $B_{\nu}(w)$ is the Bessel kernel of the third kind which satisfies for $Re\,w>0$
\begin{equation}\label{21}
|e^{\nu^2}\nu B_{\nu}(w)|\leq C|w|^{-|Re\,\nu|},\quad|w|\leq1,
\end{equation}
and
\begin{equation}\label{22}
|B_{\nu}(w)|\leq C_{Re\,\nu}e^{-Re\,w}|w|^{-1/2},\quad|w|\geq1.
\end{equation}
See \cite{KRS}, p.\,339 for details.
The key point is that the kernel $K_{\lambda}$ can be controlled
by that of the fractional integral operator $I_\alpha$
which is defined for $0<\alpha<n$ by
$$I_\alpha f(x)=\int_{\mathbb{R}^n}\frac{f(y)}{|x-y|^{n-\alpha}}dy.$$
To show this, note first that
$Re(\sqrt{z|x|^2})=|x|\cos(\frac12\arg z)>0$ for $x\neq0$,
since $-\pi<\arg z\leq\pi$ by the principal branch, and $z\not\in\mathbb{R}$.
Then, if $Re(\lambda)=(n-1)/2$, it follows from \eqref{21} that for $|x|\leq1$
\begin{align*}
|K_{\lambda}(x)|&\leq C\bigg|\frac{e^{\lambda^2}e^{-(n/2-\lambda)^2}}{n/2-\lambda}
\bigg(\frac{z}{|x|^2}\bigg)^{\frac12(\frac n2-\lambda)}e^{(n/2-\lambda)^2}(n/2-\lambda)B_{n/2-\lambda}(\sqrt{z|x|^2})\bigg|\\
&\leq C\big|z^{\frac12(\frac n2-\lambda)}|x|^{-(\frac n2-\lambda)}\big|\big|\sqrt{z|x|^2}\big|^{-1/2}\\
&\leq Ce^{Im(\lambda)\frac12\arg z}|x|^{-1}\\
&\leq Ce^{\frac\pi2|Im(\lambda)|}|x|^{-1}.
\end{align*}
On the other hand, using \eqref{22}, one has for $|x|\geq1$
\begin{align*}
|K_{\lambda}(x)|&\leq C\bigg|
\bigg(\frac{z}{|x|^2}\bigg)^{\frac12(\frac n2-\lambda)}B_{n/2-\lambda}(\sqrt{z|x|^2})\bigg|\\
&\leq C\big|z^{\frac12(\frac n2-\lambda)}|x|^{-(\frac n2-\lambda)}\big|\big|\sqrt{z|x|^2}\big|^{-1/2}\\
&\leq Ce^{\frac\pi2|Im(\lambda)|}|x|^{-1}.
\end{align*}
Hence for $Re(\lambda)=(n-1)/2$, $K_{\lambda}(x)$ is controlled by the kernel $|x|^{-1}$ of $I_{n-1}$.

Now we use the following lemma,
which characterizes weighted $L^2$ estimates for fractional integrals,
due to Kerman and Sawyer \cite{KeS} (see Theorem 2.3 there and also Lemma 2.1 in \cite{BBRV}):

\begin{lem}
Let $0<\alpha<n$. Assume that $w$ is a nonnegative measurable function on $\mathbb{R}^n$.
Then there exists a constant $C_w$ depending on $w$ such that $w\in\mathcal{KS}_\alpha$ if and only if the following two equivalent estimates
\begin{equation}\label{ee}
\|I_{\alpha/2}f\|_{L^2(w)}\leq C_w\|f\|_{L^2}
\end{equation}
and
$$\|I_{\alpha/2}f\|_{L^2}\leq C_w\|f\|_{L^2(w^{-1})}$$
are valid for all measurable functions $f$ on $\mathbb{R}^n$.
Furthermore, the constant $C_w$ may be taken to be a constant multiple of $\|w\|_{\mathcal{KS}_\alpha}^{1/2}$.
\end{lem}
First, we note that the two estimates in the lemma directly implies that
\begin{equation}\label{s2}
\big\||V|^{1/2}I_\alpha(|V|^{1/2}f)\big\|_{L^2}\leq C\|V\|_{\mathcal{KS}_\alpha}\|f\|_{L^2}.
\end{equation}
Then, since $|K_{\lambda}(x)|\leq Ce^{\frac\pi2|Im(\lambda)|}|x|^{-1}$ for $Re(\lambda)=(n-1)/2$,
using \eqref{s2} with $\alpha=n-1$, we get the desired bound \eqref{kern}.
Indeed,
\begin{align*}
\|T_\lambda f\|_{L^2}
&=\big\||V^{\frac{n-1}2}|^{1/2}(-\Delta-z)^{-\lambda}\big[|V^{\frac{n-1}2}|^{1/2}(|V|^{iIm(\lambda)/2}f)\big]\big\|_{L^2}\\
&\leq Ce^{\frac\pi2|Im(\lambda)|}
\big\||V^{\frac{n-1}2}|^{1/2}I_{n-1}\big[|V^{\frac{n-1}2}|^{1/2}(|V|^{iIm(\lambda)/2}f)\big]\big\|_{L^2}\\
&\leq Ce^{\frac\pi2|Im(\lambda)|}\|V^{\frac{n-1}2}\|_{\mathcal{KS}_{n-1}}\|f\|_{L^2}.
\end{align*}
This completes the proof.
\end{proof}


\section{Limiting absorption principle and Fourier restriction estimates}\label{sec3}

In this section we study the limiting absorption principle in spectral theory and
its relation with weighted restriction estimates in harmonic analysis.
This relationship seems to have been well known, but to the best of our knowledge,
it was treated in the literature (\cite{BRV,BBRV}) only in certain particular cases.
Here we put it in a more abstract framework for our purpose in this paper.
The resulting restriction estimates will be fundamentally used in the next section in obtaining
weighted $L^2$ estimates for the Schr\"odinger propagator.
Alternatively, we will obtain them more directly by appealing to Kato $H$-smoothing theory.

First we recall some basic notions and facts from spectral theory.
Let $T$ be a closed\footnote{\,If $T$ is not closed, then $T-z$ and $(T-z)^{-1}$ are not closed.
So, $\rho(T)$ is empty. This is why $T$ is assumed to be closed.} linear operator
on a Hilbert space $\mathcal{H}$ over $\mathbb{C}$.
Then we denote by $\rho(T)$ the resolvent set of $T$ which is the set of $z\in\mathbb{C}$
for which $T-z$ is invertible and the inverse $(T-z)^{-1}$ is a bounded operator on $\mathcal{H}$.
The spectrum $\sigma(T)=\mathbb{C}\setminus\rho(T)$ is given by the complement of the resolvent set.
One of the most fundamental facts is that the resolvent $R_T(z):=(T-z)^{-1}$ is an analytic operator-valued
function on $\rho(T)$ and the operator norm $\|R_T(z)\|$ satisfies
$$\|R_T(z)\|\geq[\,\text{dist}\,(z,\sigma(T))]^{-1}.$$
(See, for example, \cite{T}.)
The point here is that the norm of $R_T(z)$ as a map from $\mathcal{H}$ to $\mathcal{H}$
diverges as $z$ approaches $\sigma(T)$.
But there is a principle that $R_T(z)$ can remain bounded in a sense.
This is referred to as the limiting absorption principle.

Of special interest is the free resolvent which is usually denoted by $R_0(z)=(-\Delta-z)^{-1}$ on  $L^2(\mathbb{R}^n)$.
In this case, the spectrum is $[0,\infty)$ and there is a classical result due to Agmon \cite{A}
which states that the limits
$$\lim_{\varepsilon\rightarrow0}R_0(\lambda\pm i\varepsilon):=R_0(\lambda\pm i0)$$
exist for $\lambda\in(0,\infty)$ in the norm of bounded operators
from $L^2(\langle x\rangle^sdx)$ to  $L^2(\langle x\rangle^{-s}dx)$ for $s>1$.
(Here $\langle x\rangle=(1+|x|^2)^{1/2}$.)
This shows that the free resolvent can remain bounded between weighted $L^2$ spaces,
as $z$ approaches the spectrum (except $0$).
Furthermore, the principle holds in the following weak form which is called the weak limiting absorption principle:
For $f,g\in L^2(\langle x\rangle^{-s}dx)$
$$\lim_{\varepsilon\rightarrow0}\langle R_0(\lambda\pm i\varepsilon)f,g\rangle=\langle R_0(\lambda\pm i0)f,g\rangle,$$
where $\langle\text{ },\text{ }\rangle$ denotes the usual inner product on the Hilbert space $L^2(\mathbb{R}^n)$.
These principles play a key role in the study of Schr\"odinger operators as well as Helmholtz equations.

For our purpose, it is enough to consider the following weak principle:

\begin{prop}\label{wlap}
Let $n\geq2$ and let $V:\mathbb{R}^n\rightarrow\mathbb{C}$ be such that
the resolvent estimate
\begin{equation}\label{ure}
\|R_0(z)f\|_{L^2(|V|)}\leq C(V)\|f\|_{L^2(|V|^{-1})}
\end{equation}
holds uniformly in $z\in\mathbb{C}\setminus\mathbb{R}$.
Then, for $\lambda>0$ there exists the weak limit
\begin{equation}\label{wl}
\lim_{\varepsilon\rightarrow0}\langle R_0(\lambda\pm i\varepsilon)f,g\rangle:=\langle R_0(\lambda\pm i0)f,g\rangle
\end{equation}
whenever $f,g\in L^2(|V|^{-1})$.
Furthermore, $R_0(\lambda\pm i0)$ satisfies
\begin{equation}\label{di2}
\|R_0(\lambda\pm i0)f\|_{L^2(|V|)}\leq C(V)\|f\|_{L^2(|V|^{-1})}
\end{equation}
and can be given by the following distributional identity
\begin{equation}\label{di}
R_0(\lambda\pm i0)f(x)=
\pm\frac{i\pi}{2\sqrt{\lambda}}\widehat{d\sigma_{\sqrt{\lambda}}}\ast f(x)+
p.v.\int_{\mathbb{R}^n}e^{ix\cdot\xi}\frac{\widehat{f}(\xi)}{|\xi|^2-\lambda}d\xi.
\end{equation}
Here, $d\sigma_r$ denotes the induced Lebesgue measure on the sphere $S_r^{n-1}$ in $\mathbb{R}^n$ with radius $r$,
and the singular integral in \eqref{di} is taken in the principal value sense.
\end{prop}

Assuming for the moment this proposition, we first explain how to deduce Fourier restriction estimates
from the weak principle.
From \eqref{di2} and taking the imaginary part of the operator $R_0(\lambda\pm i0)$ in \eqref{di},
it follows that
\begin{equation}\label{dl}
\|\widehat{d\sigma_r}\ast f\|_{L^2(|V|)}\leq rC(V)\|f\|_{L^2(|V|^{-1})}.
\end{equation}
Then, using this and the standard $TT^\ast$ argument of Tomas \cite{To},
we see that
\begin{align*}
\int_{S_r^{n-1}}|\widehat{f}|^2d\sigma=\int_{S_r^{n-1}}\widehat{f}\overline{\widehat{f}}d\sigma
&=\int_{\mathbb{R}^n}f(f\ast\widehat{d\sigma_r})dx\\
&\leq\|f\|_{L^2(|V|^{-1})}\|f\ast\widehat{d\sigma_r}\|_{L^2(|V|)}\\
&\leq rC(V)\|f\|_{L^2(|V|^{-1})}^2.
\end{align*}
Namely, we get the following weighted $L^2$ restriction estimate
$$\|\widehat{f}\|_{L^2(S_r^{n-1})}\leq r^{1/2}C(V)^{1/2}\|f\|_{L^2(|V|^{-1})}.$$
By duality this is equivalent to
$$\|\widehat{fd\sigma_r}\|_{L^2(|V|)}\leq r^{1/2}C(V)^{1/2}\|f\|_{L^2(S_r^{n-1})}.$$
The first estimate \eqref{dl} can be also deduced immediately from these two equivalent estimates,
and what we have just explained is summarized in the following corollary.

\begin{cor}\label{rest}
Let $n\geq2$ and let $V:\mathbb{R}^n\rightarrow\mathbb{C}$ be such that the resolvent estimate \eqref{ure} holds.
Then the following three equivalent estimates hold:
\begin{equation*}
\|\widehat{d\sigma_r}\ast f\|_{L^2(|V|)}\leq rC(V)\|f\|_{L^2(|V|^{-1})},
\end{equation*}
\begin{equation*}
\|\widehat{f}\|_{L^2(S_r^{n-1})}\leq r^{1/2}C(V)^{1/2}\|f\|_{L^2(|V|^{-1})},
\end{equation*}
\begin{equation}\label{43}
\|\widehat{fd\sigma_r}\|_{L^2(|V|)}\leq r^{1/2}C(V)^{1/2}\|f\|_{L^2(S_r^{n-1})}.
\end{equation}
\end{cor}

\begin{rem}
Given a class $\mathcal{R}$ of resolvent type, this corollary clearly holds for $V\in\mathcal{R}$ with $C(V)\sim[V]_{\mathcal{R}}$.
In the cases of $\mathcal{R}=\mathcal{L}^{2,p},\,p>(n-1)/2$, and $\mathcal{S}_3$,
the above restriction estimates can be found in \cite{CS} and \cite{BBRV}, respectively.
In view of the resolvent estimates in Theorem \ref{thm2.3},
these previous results are now extended to the class $\mathcal{S}_n$.
\end{rem}

\begin{proof}[Proof of Proposition \ref{wlap}]
We basically follow the argument of Agmon \cite{A}.
In fact the argument given by him was for the case $V(x)=\langle x\rangle^{-s}$, $s>1$,
but it goes through in our abstract setting.

Let us first consider the analytic function $F(z)=\langle R_0(z)f,g\rangle$ for $z\in\mathbb{C}\setminus\mathbb{R}$.
Then it follows from \eqref{ure} that
$$|\langle R_0(z)f,g\rangle|\leq C(V)\|f\|_{L^2(|V|^{-1})}\|g\|_{L^2(|V|^{-1})}.$$
Since this estimate is uniform in $z$ and $C_0^\infty$ is dense in $L^2(|V|^{-1})$,
by a standard limiting argument, it suffices to show \eqref{wl} for $f,g\in C_0^\infty$.
Here, to see that $C_0^\infty$ is dense in $L^2(|V|^{-1})$,
consider first $D_n=\{x\in\mathbb{R}^n:|V|^{-1}\geq1/n\}$.
Then, for $f\in L^2(|V|^{-1})$ the function $\chi_{D_n}f$ is contained in $L^2\cap L^2(|V|^{-1})$,
and $\chi_{D_n}f|V|^{-1/2}\rightarrow f|V|^{-1/2}$ as $n\rightarrow\infty$.
By the Lebesgue dominated convergence theorem,
it follows now that $\chi_{D_n}f|V|^{-1/2}\rightarrow f|V|^{-1/2}$ in $L^2$.
This means that $L^2\cap L^2(|V|^{-1})$ is dense in $L^2(|V|^{-1})$.
On the other hand, it is easy to see that $C_0^\infty$ is dense in $L^2\cap L^2(|V|^{-1})$
with respect to the norm $L^2(|V|^{-1})$ using a $C_0^\infty$ approximate identity.
Consequently, $C_0^\infty$ is dense in $L^2(|V|^{-1})$.

Now, using Parseval's formula and changing to polar coordinates, we see that
\begin{align*}
\nonumber\langle R_0(z)f,g\rangle&=
\int_{\mathbb{R}^n}\frac{\widehat{f}(\xi)\overline{\widehat{g}(\xi)}}{|\xi|^2-z}d\xi\\
&=\frac12\int_0^\infty\frac{t^{(n-2)/2}}{t-z}
\bigg(\int_{|\omega|=1}\widehat{f}(\sqrt{t}\omega)\overline{\widehat{g}(\sqrt{t}\omega)}d\omega\bigg)dt.
\end{align*}
Thus, by a well-known continuity property of Cauchy type integrals
(Sokhotsky-Plemelj formula), $\langle R_0(z)f,g\rangle$
has continuous boundary values given by
\begin{equation}\label{form}
\lim_{z\rightarrow\lambda,\,\pm Imz>0}\langle R_0(z)f,g\rangle
=\pm\frac{\pi i}{2\sqrt{\lambda}}
\int_{|\xi|=\sqrt{\lambda}}\widehat{f}(\xi)\overline{\widehat{g}(\xi)}d\sigma
+p.v.\int_{\mathbb{R}^n}\frac{\widehat{f}(\xi)\overline{\widehat{g}(\xi)}}{|\xi|^2-\lambda}d\xi
\end{equation}
on both sides of $(0,\infty)$ (see (4.7) in \cite{A}).
This means that for $f\in L^2(|V|^{-1})$ there exists the weak limit
$$w-\lim_{z\rightarrow\lambda,\,\pm Imz>0}R_0(z)f:=R_0(\lambda\pm i0)f$$
in $L^2(|V|)$.
Namely, for $f,g\in L^2(|V|^{-1})$
$$\lim_{z\rightarrow\lambda,\,\pm Imz>0}\langle R_0(z)f,g\rangle=\langle R_0(\lambda\pm i0)f,g\rangle,$$
and the identity ~\eqref{di} follows immediately from ~\eqref{form}.
Now, by duality the estimate \eqref{di2} is deduced from \eqref{wl} and \eqref{ure}.
Indeed, note that
\begin{align*}
\|R_0(\lambda\pm i0)f\|_{L^2(|V|)}
&=\sup_{\|\widetilde{g}\|_{L^2}\leq1}|\langle|V|^{1/2}R_0(\lambda\pm i0)f,\widetilde{g}\rangle|\\
&=\sup_{\|\widetilde{g}\|_{L^2}\leq1}|\langle R_0(\lambda\pm i0)f,|V|^{1/2}\widetilde{g}\rangle|\\
&=\sup_{\|g\|_{L^2(|V|^{-1})}\leq1}|\langle R_0(\lambda\pm i0)f,g\rangle|,
\end{align*}
where $g=|V|^{1/2}\widetilde{g}$.
By \eqref{wl} and \eqref{ure}, this readily leads to
$$\| R_0(\lambda\pm i0)f\|_{L^2(|V|)}\leq C(V)\|f\|_{L^2(|V|^{-1})}.$$
This completes the proof.
\end{proof}


\section{Weighted $L^2$ estimates for the Schr\"odinger propagator}\label{sec4}

Let us first consider the following initial value problem for the free Schr\"odinger equation:
\begin{equation*}
\left\{
\begin{array}{ll}
i\partial_tu+\Delta u=0,\\
u(x,0)=u_0(x).
\end{array}\right.
\end{equation*}
In view of the Fourier transform, the solution is explicitly given by
\begin{equation}\label{du}
u(x,t)=e^{it\Delta}u_0(x)=(2\pi)^{-n}\int_{\mathbb{R}^n} e^{ix\cdot\xi}e^{-it|\xi|^2}\widehat{u_0}(\xi)d\xi,
\end{equation}
where the evolution operator $e^{it\Delta}$ is called the free Schr\"odinger propagator.
Then the upshot of this section is the following weighted $L^2$ estimate for the propagator,
which will play a key role in the next section in obtaining our abstract Carleman estimates.

\begin{prop}\label{istr}
Let $n\geq2$ and let $V:\mathbb{R}^n\rightarrow\mathbb{C}$ be such that the resolvent estimate
\begin{equation}\label{ure0}
\|R_0(z)f\|_{L^2(|V|)}\leq C(V)\|f\|_{L^2(|V|^{-1})}
\end{equation}
holds uniformly in $z\in\mathbb{C}\setminus\mathbb{R}$.
Then we have
\begin{equation}\label{up}
\bigg\|\int_{-\infty}^\infty e^{i(t-s)\Delta}F(\cdot,s)ds\bigg\|_{L_{t,x}^2(|V|)}
\leq CC(V)\|F\|_{L_{t,x}^2(|V|^{-1})},
\end{equation}
where $F:\mathbb{R}^{n+1}\rightarrow\mathbb{C}$ is a function in $L_{t,x}^2(|V|^{-1})$.
\end{prop}

The basic approach for this is to obtain the following estimate
\begin{equation}\label{homost}
\big\|e^{it\Delta}f\big\|_{L_{t,x}^2(|V|)}\leq CC(V)^{1/2}\|f\|_2
\end{equation}
which may be referred to as the weighted $L^2$ homogeneous Strichartz estimate.
(In Section \ref{sec8} we will also obtain the corresponding inhomogeneous estimate to study
the well-posedness of the Schr\"odinger equation.)
Indeed, by duality\footnote{Consider the operator $Tf(x,t)=|V(x)|^{1/2}e^{it\Delta}f(x)$. 
Then, $\|Tf\|_{L_{t,x}^2}\leq CC(V)^{1/2}\|f\|_2$ from \eqref{homost}.
Since the adjoint operator of $T$ is now given as 
$T^\ast G(x)=\int_{-\infty}^\infty e^{-is\Delta}(|V(\cdot)|^{1/2}G(\cdot,s))ds$,
by duality, $\|T^\ast G\|_2\leq CC(V)^{1/2}\|G\|_{L_{t,x}^2}$ which is equivalent to \eqref{dual}.} 
this is equivalent to
\begin{equation}\label{dual}
\bigg\|\int_{-\infty}^\infty e^{-is\Delta}F(\cdot,s)ds\bigg\|_2
\leq CC(V)^{1/2}\|F\|_{L_{t,x}^2(|V|^{-1})},
\end{equation}
and so one gets \eqref{up} combining \eqref{homost} and \eqref{dual}.

In this regard, we aim at proving the following lemma.
We shall give two proofs of that.
The first is a concrete one using the weighted $L^2$ restriction estimate given in the previous section.
On the other hand, the second is a more direct one appealing to Kato $H$-smoothing theory.

\begin{lem}\label{str0}
Let $n\geq2$ and let $V:\mathbb{R}^n\rightarrow\mathbb{C}$ be such that
the resolvent estimate \eqref{ure0} holds.
Then we have
\begin{equation}\label{homostr1}
\big\|e^{it\Delta}f\big\|_{L_{t,x}^2(|V|)}\leq CC(V)^{1/2}\|f\|_2.
\end{equation}
\end{lem}

\begin{rem}
The estimate \eqref{homostr1} was first obtained implicitly in \cite{RV2}
for Fefferman-Phong potentials $V\in\mathcal{L}^{2,p}$, $p>(n-1)/2$, $n\geq3$,
and this was extended in \cite{BBRV} to the Kerman-Saywer class $\mathcal{K}_2$ only for $n=3$.
By combining the resolvent estimates in Theorem \ref{thm2.3} and this lemma,
we can extend these previous results to the class $\mathcal{S}_n$.
\end{rem}

\begin{proof}[Proof of Lemma \ref{str0}]
The proof is quite standard and based on the weighted $L^2$ restriction estimate given in the previous section.
Indeed, using polar coordinates and changing variables $r^2=\lambda$, one can see that
\begin{align*}
e^{it\Delta}f&=\int_0^\infty e^{-itr^2}\int_{S_r^{n-1}}e^{ix\cdot\xi}\widehat{f}(\xi)d\sigma_r(\xi)dr\\
&=\frac12\int_0^\infty e^{-it\lambda}\int_{S_{\sqrt{\lambda}}^{n-1}}
e^{ix\cdot\xi}\widehat{f}(\xi)d\sigma_{\sqrt{\lambda}}(\xi)\lambda^{-1/2}d\lambda.
\end{align*}
By Plancherel's theorem in $t$, it follows now that
\begin{align*}
\big\|e^{it\Delta}f\big\|_{L_{t,x}^2(|V|)}^2&\leq
C\int_{\mathbb{R}^n}\bigg(\int_0^\infty\bigg|\int_{S_{\sqrt{\lambda}}^{n-1}}
e^{ix\cdot\xi}\widehat{f}(\xi)d\sigma_{\sqrt{\lambda}}(\xi)\bigg|^2\lambda^{-1}d\lambda\bigg)|V(x)|dx\\
&\leq C\int_0^\infty\bigg(\int_{\mathbb{R}^n}
\bigg|\int_{S_r^{n-1}}e^{ix\cdot\xi}\widehat{f}(\xi)d\sigma_r(\xi)\bigg|^2|V(x)|dx\bigg)r^{-1}dr.
\end{align*}
Combining this and the estimate \eqref{43} in Corollary \ref{rest}, one gets
\begin{align*}
\big\|e^{it\Delta}f\big\|_{L_{t,x}^2(|V|)}^2&\leq
CC(V)\int_0^\infty\int_{S_r^{n-1}}|\widehat{f}(\xi)|^2d\sigma_r(\xi)dr\\
&=CC(V)\|f\|_2^2
\end{align*}
as desired.
\end{proof}

Now we derive \eqref{homostr1} more directly from the resolvent estimate
appealing to Kato $H$-smoothing theory.
The notion of $H$-smoothing due to Kato \cite{K} was first appeared in the context of scattering theory,
and its usefulness for dispersive equations was revealed in some recent works (\cite{RoS,D'AF}).
We shall use a version of the notion to suit our purpose.

Let $H$ be a self-adjoint operator in a Hilbert space $\mathcal{H}$,
whose resolvent is $R_H(z)=(H-z)^{-1}$ for $z\in\mathbb{C}\setminus\mathbb{R}$.
Let $T$ be a densely defined, closed operator on $\mathcal{H}$ with domain $D(T)$.
Then the following lemma due to Kato \cite{K} (see also \cite{RS3}, XIII.7)
allows us to employ the resolvent estimate for $H$ to obtain a space-time estimate for the Schr\"odinger propagator.

\begin{lem}\label{lem4.4}
Assume that there exists a constant $\widetilde{C}>0$ such that for $g\in D(T^\ast)$
\begin{equation}\label{ksm}
\sup_{z\in\mathbb{C}\setminus\mathbb{R}}\|TR_H(z)T^\ast g\|_{\mathcal{H}}\leq\widetilde{C}\|g\|_{\mathcal{H}}.
\end{equation}
Then the operator $T$ is $H$-smooth, i.e.,
$e^{itH}f\in D(T)$ for $f\in\mathcal{H}$ and almost every $t$, and
$$\int_{\mathbb{R}}\|Te^{itH}f\|_{\mathcal{H}}^2dt\leq\widetilde{C}\|f\|_{\mathcal{H}}^2.$$
\end{lem}

\begin{proof}[Alternate proof of Lemma \ref{str0}]
Now we appeal to this lemma in order to prove Lemma \ref{str0}.
By replacing $f=|V|^{1/2}g$ in the resolvent estimate \eqref{ure0} we see that
\begin{equation*}
\sup_{z\in\mathbb{C}\setminus\mathbb{R}}\||V|^{1/2}R_0(z)|V|^{1/2}g\|_{L^2}\leq C(V)\|g\|_{L^2}.
\end{equation*}
Hence, applying Lemma \ref{lem4.4} with  $H=-\Delta$, $\mathcal{H}=L^2$, and $T:f\mapsto|V|^{1/2}f$,
we get immediately the desired estimate \eqref{homostr1}.
In doing so, it seems to be more or less complicated to check that the multiplication operator
$T:f\mapsto|V|^{1/2}f$ is a densely defined, closed operator on $L^2$
with $D(T)=\{f\in L^2:Tf\in L^2\}$.
Let us now check this.
First, it is easy to see that $D(T)$ is dense in $L^2$.
In fact, consider $D_n=\{x\in\mathbb{R}^n:|V|^{1/2}\leq n\}$.
Then, for $f\in L^2$ the function $\chi_{D_n}f$ is contained in $D(T)$,
and $\chi_{D_n}f\rightarrow f$ as $n\rightarrow\infty$.
By the Lebesgue dominated convergence theorem, it follows now that $\chi_{D_n}f\rightarrow f$ in $L^2$.
So, $D(T)$ is dense in $L^2$.
Next, note that since $|V|^{1/2}$ is real, $T$ is trivially self-adjoint, i.e., $T=T^{\ast}$,
and so $T^\ast$ is also dense in $L^2$.
From these facts, $T$ is closable and its closure can be given by $\overline{T}=T^{\ast\ast}$.
(See, for example, Theorem VIII.1 in \cite{RS}.)
Thus, $\overline{T}=T^{\ast\ast}=T^\ast=T$. That is, $T$ is closed.
\end{proof}

\section{Carleman estimates}\label{sec5}

This section is devoted to proving the following abstract Carleman estimate
which means that a potential for the resolvent estimate
can be made to work for the Carleman estimate in an abstract way.

\begin{prop}\label{prop}
Let $n\geq2$ and
let $[V]$ be the least constant $C(V)$ for which the resolvent estimate
\begin{equation}\label{ure55}
\|R_0(z)f\|_{L^2(|V|)}\leq C(V)\|f\|_{L^2(|V|^{-1})}
\end{equation}
holds uniformly in $z\in\mathbb{C}\setminus\mathbb{R}$.
Then we have for $u\in C_0^\infty(\mathbb{R}^{n+1})$
\begin{equation}\label{Carl}
\big\|e^{\beta\langle(x,t),\nu\rangle}u\big\|_{L_{t,x}^2(|V(x)|)} \leq
C[V]\big\|e^{\beta\langle(x,t),\nu\rangle}(i\partial_t+\Delta)u\big\|_
{L_{t,x}^2(|V(x)|^{-1})}
\end{equation}
with a constant $C$ independent of $\beta\in\mathbb{R}$ and $\nu\in\mathbb{R}^{n+1}$,
if $|V|\in A_2(\nu^\prime)$.
Here, recall that $\nu^\prime=(\nu_1,...,\nu_n)$ for $\nu=(\nu_1,...,\nu_n,\nu_{n+1})$.
\end{prop}

\begin{rem}
The assumption $|V|\in A_2(\nu^\prime)$
is the only one needed in order to derive \eqref{Carl} from \eqref{ure55}.
Once a suitable class $\mathcal{R}$ is established,
\eqref{Carl} holds for $V\in\mathcal{R}$ with $[V]$ replaced by $[V]_{\mathcal{R}}$.
In this concrete setting there are some classes of $\mathcal{R}$
where the $A_2$ assumption is superfluous.
See a comment below the proof of Proposition \ref{pr00} for details.
\end{rem}

To obtain \eqref{Carl}, we first convert it into the following form
which seems at first glance to be easier to handle:
\begin{equation}\label{sobol}
\|u\|_{L_{t,x}^2(|V(x)|)}\leq C[V]\|(i\partial_t+\Delta+P(D))u\|_{L_{t,x}^2(|V(x)|^{-1})},
\end{equation}
where $P(D)=-2\beta\langle\nu',\nabla\rangle+\beta^2|\nu'|^2-i\beta\nu_{n+1}$
arises from a computation of the conjugated operator $e^{\beta\langle(x,t),\nu\rangle}(i\partial_t+\Delta)e^{-\beta\langle(x,t),\nu\rangle}$.
Note that $e^{\beta\langle(x,t),\nu\rangle}$ no longer appears explicitly in this form
although it affects $P(D)$.
The key point is that $P(D)$ can be controlled uniformly in $\beta$, $\nu$ by decomposing
$(i\partial_t+\Delta+P(D))^{-1}$ in the phase space,
along with the corresponding direction of the vector $\nu^\prime$,
into a number of localized pieces which can be made in a sensible way to behave like the resolvent $(-\Delta-z)^{-1}$.
These pieces will then be estimated and recombined successfully on weighted $L^2$ spaces.
The $A_2$ assumption comes into play at this step.

In fact we obtain the Sobolev type inequality \eqref{sobol} in the following more general setting:

\begin{prop}\label{pr00}
Let $n\geq2$ and let $P(D)$ be a first-order differential operator given by
$P(D)=\langle \vec{c},\nabla\rangle+z$,
where $\vec{c}=\vec{a}+i\vec{b}\in\mathbb{C}^n$ with $\vec{a},\vec{b}\in\mathbb{R}^n$,
and $z\in\mathbb{C}$.
Then we have for $u\in C_0^\infty(\mathbb{R}^{n+1})$
\begin{equation}\label{sob000}
\|u\|_{L_{t,x}^2(|V(x)|)}\leq C[V]
\|(i\partial_t+\Delta+P(D))u\|_{L_{t,x}^2(|V(x)|^{-1})}
\end{equation}
with a constant $C$ independent of $\vec{c}$ and $z$, if $|V|\in A_2(\vec{a})$.
\end{prop}

It had been conjectured that the following estimate holds for some $p\neq2$:
\begin{equation}\label{sob22}
\|u\|_{L_{t,x}^p}\leq C\|(i\partial_t+\Delta+i)u\|_{L_{t,x}^p}.
\end{equation}
Note that the multiplier $1/(-\tau-|\xi|^2+i)$ associated with $(i\partial_t+\Delta+i)^{-1}$ is bounded,
and so \eqref{sob22} is trivially satisfied for $p=2$ by Plancherel's theorem.
Interest in \eqref{sob22} came from the work of Calder\'{o}n \cite{C} concerning the $L^p$ boundedness
of Fourier multipliers given by bounded rational functions,
and it was shown in \cite{KeT} that the conjecture is false, namely, \eqref{sob22} holds only for $p=2$.
Our estimate \eqref{sob000} can be viewed as a weighted $L^2$ version,
with the general first-order term $P(D)$, of \eqref{sob22}.

\begin{proof}[Proof of Proposition \ref{pr00}]
First, note that $[V]$ is invariant under translations $T$ and rotations $R$.
That is, $[V(Tx)]=[V]$ and $[V(Rx)]=[V]$.
It is an elementary matter to check this.
Of course, $[CV(x)]=C[V]$ for constants $C>0$.
Then by elementary arguments, the inequality \eqref{sob000} can be reduced to the following particular case where $P(D)=\partial_{x_n}$:
\begin{equation}\label{special0}
\|u\|_{L_{t,x}^2(|V|)}\leq C[V]
\|(i\partial_t+\Delta+\partial_{x_n})u\|_{L_{t,x}^2(|V|^{-1})}.
\end{equation}
In fact, by setting $u=e^{-i\langle \vec{b}/2,x\rangle}v$, $\vec{b}\in\mathbb{R}^n$,
it follows that
$$(i\partial_t+\Delta+P(D))u=e^{-i\langle \vec{b}/2,x\rangle}
(i\partial_t+\Delta+\langle \vec{c}-i\vec{b},\nabla\rangle+z-i\langle \vec{c},\vec{b}/2\rangle-|\vec{b}/2|^2)v.$$
Hence one may assume that $\vec{c}\in\mathbb{R}^n$.
Similarly, for $z=a+ib\in\mathbb{C}$, by setting $u=e^{iat}v$, one can see that
$$(i\partial_t+\Delta+P(D))u=e^{iat}(i\partial_t+\Delta+\langle \vec{c},\nabla\rangle+ib)v.$$
From this one may also assume that $Re\,z=0$.
Since the Laplacian $\Delta$ is trivially invariant under rotations and so is $[V]$,
by a simple rotation argument the assumption $|V|\in A_2(\vec{a})$ is reduced to
the case where $|V|\in A_2(\mathbb{R})$ in the $x_n$ variable uniformly in other variables,
and $\langle \vec{c},\nabla\rangle=c\partial_{x_n}$ with $c\in\mathbb{R}$.
So far, we have explained how to reduce the matter to the case where $P(D)=c\partial/\partial x_n+ib$
with $c,b\in\mathbb{R}$.
For simplicity of notation, we shall also assume that $c=1$ and $b=0$,
because it does not affect all the arguments in the proof.

Our basic plan for \eqref{special0} is to decompose
the inverse $(i\partial_t+\Delta+\partial_{x_n})^{-1}$ in the Fourier transform side,
along with the $\xi_n$ axis, into a number of localized pieces
which can be made to behave like the resolvent $(-\Delta-z)^{-1}$,
and to estimate and recombine them on weighted $L^2$ spaces
using Proposition \ref{istr} and appealing to real variable theory in \cite{Ku}.

Let us now rewrite \eqref{special0} in the Fourier transform side as
\begin{equation}\label{multi}
\bigg\|\mathcal{F}^{-1}\bigg(\frac{\widehat{f}(\xi,\tau)}{\tau+|\xi|^2+i\xi_n}\bigg)\bigg\|_{L_{t,x}^2(|V|)}
\leq C[V]\|f\|_{L_{t,x}^2(|V|^{-1})}
\end{equation}
for $f\in C_0^\infty(\mathbb{R}^{n+1})$.
Setting $m(\xi,\tau)=(\tau+|\xi|^2+i\xi_n)^{-1}$, we define the multiplier operator $\mathcal{T}$ by
$$\widehat{\mathcal{T}f}(\xi,\tau)=m(\xi,\tau)\widehat{f}(\xi,\tau).$$
For $k\in\mathbb{Z}$, we also set $\phi_k(r)=\phi(2^kr)$ for $\phi\in C_0^\infty(\mathbb{R})$
which is such that $\phi(r)=1$ if $|r|\sim1$, and $\phi(r)=0$ otherwise.
Using this, we decompose $\mathcal{T}$ into the localized operators $\mathcal{T}_k$, $k\in\mathbb{Z}$,
which are given by
$$\widehat{\mathcal{T}_kf}(\xi,\tau)=m(\xi,\tau)\chi_k(\xi_n)\widehat{f}(\xi,\tau).$$
Now, \eqref{multi} becomes equivalent to
\begin{equation}\label{multi00}
\big\|\sum_{k\in\mathbb{Z}}\mathcal{T}_kf\big\|_{L_{t,x}^2(|V|)}
\leq C[V]\|f\|_{L_{t,x}^2(|V|^{-1})}.
\end{equation}
To show this, we assume for the moment that
\begin{equation}\label{multi3}
\|\mathcal{T}_kf\|_{L_{t,x}^2(|V|)}\leq C[V]\|f\|_{L_{t,x}^2(|V|^{-1})}
\end{equation}
with $C$ independent of $k\in\mathbb{Z}$.
Then, by the Littlewood-Paley theorem on weighted $L^2$ spaces with weights in the $A_2$ class
(see Theorem 1 in \cite{Ku}), we can get the desired estimate \eqref{multi00}.
Indeed, since $|V|\in A_2(\mathbb{R})$ in the $x_n$ variable uniformly in other variables,
by the Littlewood-Paley theorem in $x_n$ and using \eqref{multi3} it follows that
\begin{align}\label{wrf}
\nonumber\big\|\sum_k\mathcal{T}_kf\big\|_{L_{t,x}^2(|V|)}
&\leq C\bigg\|\bigg(\sum_{k\in\mathbb{Z}}|\mathcal{T}_kf|^2\bigg)^{1/2}\bigg\|_{L_{t,x}^2(|V|)}\\
\nonumber&=C\bigg(\iint\sum_{k\in\mathbb{Z}}|\mathcal{T}_kf|^2|V| dxdt\bigg)^{1/2}\\
\nonumber&=C\bigg(\sum_{k\in\mathbb{Z}}\|\mathcal{T}_kf\|_{L_{t,x}^2(|V|)}^2\bigg)^{1/2}\\
&\leq C[V]\bigg(\sum_{k\in\mathbb{Z}}\|f_k\|_{L_{t,x}^2(|V|^{-1})}^2\bigg)^{1/2},
\end{align}
where $\widehat{f_k}(\xi,\tau)=\chi_k(\xi_n)\widehat{f}(\xi,\tau)$.
On the other hand, since $|V|\in A_2(\mathbb{R})$ if and only if $|V|^{-1}\in A_2(\mathbb{R})$,
by the Littlewood-Paley theorem again, we see that
\begin{align}\label{poi}
\nonumber\bigg(\sum_{k\in\mathbb{Z}}\|f_k\|_{L_{t,x}^2(|V|^{-1})}^2\bigg)^{1/2}
&=\bigg(\iint\sum_{k\in\mathbb{Z}}|f_k|^2|V|^{-1}dxdt\bigg)^{1/2}\\
\nonumber&=\bigg\|\bigg(\sum_{k\in\mathbb{Z}}|f_k|^2\bigg)^{1/2}\bigg\|_{L_{t,x}^2(|V|^{-1})}\\
&\leq C\|f\|_{L_{t,x}^2(|V|^{-1})}.
\end{align}
Combining \eqref{wrf} and \eqref{poi}, we get the desired estimate \eqref{multi00}.

Now it remains to show the estimate \eqref{multi3}.
Equivalently, we have to show that \eqref{multi} holds for $f\in C_0^\infty(\mathbb{R}^{n+1})$
satisfying $\text{supp}\,\widehat{f}\subset\{(\xi,\tau)\in\mathbb{R}^{n+1}:|\xi_n|\sim2^{-k}\}$,
with the constant $C$ independent of $k\in\mathbb{Z}$.
To do so, we first derive the following bound from the resolvent estimate:
\begin{equation}\label{multi2}
\bigg\|\mathcal{F}^{-1}\bigg(\frac{\widehat{f}(\xi,\tau)}
{\tau+|\xi|^2+i2^{-k}}\bigg)\bigg\|_{L_{t,x}^2(|V|)}
\leq [V]\|f\|_{L_{t,x}^2(|V|^{-1})}.
\end{equation}
In fact, note that
\begin{align*}
\bigg\|\mathcal{F}^{-1}\bigg(\frac{\widehat{f}(\xi,\tau)}
{\tau+|\xi|^2+i2^{-k}}&\bigg)\bigg\|_{L_{t,x}^2(|V|)}^2\\
=&\int |V|\int\bigg|\int e^{it\tau}\bigg(\int e^{ix\cdot\xi}
\frac{\widehat{f}(\xi,\tau)}{\tau+|\xi|^2+i2^{-k}}d\xi\bigg)d\tau\bigg|^2dtdx\\
=&\int |V|\int\bigg|\int e^{ix\cdot\xi}
\frac{\widehat{f}(\xi,\tau)}{\tau+|\xi|^2+i2^{-k}}d\xi\bigg|^2d\tau dx\\
=&\iint\bigg|\int e^{ix\cdot\xi}
\frac{\widehat{f}(\xi,\tau)}{\tau+|\xi|^2+i2^{-k}}d\xi\bigg|^2|V| dxd\tau.
\end{align*}
Here we used Plancheral's theorem in $t$ for the second equality.
Applying the resolvent estimate \eqref{ure55} with $z=-\tau-i2^{-k}$ to the last term in the above
and using Plancheral's theorem in $\tau$,
we get
\begin{align*}
\bigg\|\mathcal{F}^{-1}\bigg(\frac{\widehat{f}(\xi,\tau)}
{\tau+|\xi|^2+i2^{-k}}\bigg)\bigg\|_{L_{t,x}^2(|V|)}^2
&\leq[V]^2\int\big\|\widehat{f(x,\cdot)}(\tau)\big\|_{L_{x}^2(|V|^{-1})}^2d\tau\\
&\leq [V]^2\|f\|_{L_{t,x}^2(|V|^{-1})}^2
\end{align*}
as desired.
Next, we will obtain
\begin{equation}\label{cla}
\bigg\|\mathcal{F}^{-1}\bigg(\frac{(2^{-k}-\xi_n)\widehat{f}(\xi,\tau)}
{(\tau+|\xi|^2+i\xi_n)(\tau+|\xi|^2+i2^{-k})}\bigg)\bigg\|_{L_{t,x}^2(|V|)}
\leq C[V]\|f\|_{L_{t,x}^2(|V|^{-1})}
\end{equation}
under the condition that $\text{supp}\,\widehat{f}\subset\{(\xi,\tau)\in\mathbb{R}^{n+1}:|\xi_n|\sim2^{-k}\}$.
Then, by noting that
$$\frac{1}{\tau+|\xi|^2+i\xi_n}=\frac{1}{\tau+|\xi|^2+i2^{-k}}+
i\frac{(2^{-k}-\xi_n)\widehat{f}(\xi,\tau)}{(\tau+|\xi|^2+i\xi_n)(\tau+|\xi|^2+i2^{-k})},$$
the above two estimates \eqref{multi2} and \eqref{cla} would imply the desired estimate \eqref{multi3}.

For \eqref{cla} we make use of the weighted $L^2$ estimate for the Schr\"odinger propagator
in Proposition \ref{istr}.
By changing variables $\tau+|\xi|^2\rightarrow\rho$ in \eqref{cla},
we need to show that
\begin{equation}\label{multi4}
\bigg\|\int_{\mathbb{R}}e^{it\rho}\int_{\mathbb{R}}e^{i(t-s)\Delta}F(\cdot,s)dsd\rho\bigg\|_{L_{t,x}^2(|V|)}
\leq C[V]\|f\|_{L_{t,x}^2(|V|^{-1})},
\end{equation}
where
$$\widehat{F(\cdot,s)}(\xi)=\frac{(2^{-k}-\xi_n)}{(\rho+i\xi_n)(\rho+i2^{-k})}e^{-is\rho}\widehat{f(\cdot,s)}(\xi).$$
Using Minkowski's inequality and Proposition \ref{istr},
the left-hand side of \eqref{multi4} is bounded by
\begin{equation}\label{dnff}
C[V]\int_{\mathbb{R}}\bigg\|\mathcal{F}^{-1}
\bigg(\frac{2^{-k}-\xi_n}{(\rho+i\xi_n)(\rho+i2^{-k})}e^{-is\rho}\widehat{f(\cdot,s)}(\xi)\bigg)\bigg\|_
{L_{t,x}^2(|V|^{-1})}d\rho.
\end{equation}
Now, let us set
$$m_{k,\rho}(\xi_n)=\frac{2^{-k}-\xi_n}{(\rho+i\xi_n)(\rho+i2^{-k})}.$$
Then it is an elementary matter to check
$$|m_{k,\rho}(\xi_n)|\leq \frac{C2^{-k}}{\rho^2+2^{-2k}}$$
and
$$\bigg|\frac{\partial m_{k,\rho}(\xi_n)}{\partial\xi_n}\bigg|\leq \frac{C}{\rho^2+2^{-2k}},$$
since we are assuming that $\text{supp}\,\widehat{f}\subset\{(\xi,\tau)\in\mathbb{R}^{n+1}:|\xi_n|\sim2^{-k}\}$.
So, $m_{k,\rho}(\xi_n)$ satisfies the conditions of the Marcinkiewicz multiplier theorem.
Since $|V|^{-1}\in A_2(\mathbb{R})$ in the $x_n$ variable uniformly in other variables,
by the multiplier theorem on weighted $L^2$ spaces with weights in the $A_2$ class
(see Theorem 2 in \cite{Ku}), \eqref{dnff} is bounded by
\begin{equation*}
C[V]\int_{\mathbb{R}}\frac{2^{-k}}{\rho^2+2^{-2k}}\|f\|_{L_{t,x}^2(|V|^{-1})}d\rho.
\end{equation*}
Consequently we get \eqref{multi4} since $2^{-k}/(\rho^2+2^{-2k})\in L^1$ uniformly in $k\in\mathbb{Z}$.
This completes the proof.
\end{proof}

\noindent\textbf{Comment on the $A_2$ assumption.}
Given a class $\mathcal{R}$ of resolvent type, the Sobolev type inequality \eqref{sob000} clearly holds
for $V\in\mathcal{R}$ with $[V]$ replaced by $[V]_{\mathcal{R}}$.
We point out that the $A_2$ assumption in \eqref{sob000} is superfluous in the cases of
$\mathcal{R}=L^{n/2},\,L^{n/2,\infty}$, or $\mathcal{L}^{2,p},\,p>(n-1)/2$,
and so the Carleman estimate does not also need that assumption in these cases.
Since $[V]_{\mathcal{R}}$ is given by the usual norm for these cases,
we only need to consider the case of $\mathcal{R}=\mathcal{L}^{2,p}$, $p>(n-1)/2$, because
$\|V\|_{L^{n/2}}\sim\|V\|_{\mathcal{L}^{2,n/2}}$
and $\|V\|_{\mathcal{L}^{2,p}}\leq C\|V\|_{L^{n/2,\infty}}$.
(Recall that $L^{n/2}=\mathcal{L}^{2,p}$ for $p=n/2$ and $L^{n/2,\infty}\subset\mathcal{L}^{2,p}$ if $p<n/2$.)
In this case, due to a good property of $\mathcal{L}^{2,p}$,
more tools are available in removing the $A_2$ assumption ({\it cf. \cite{CS,S3}}).

First, recall from \eqref{ap} the definition of $A_p$ weights for $1<p<\infty$.
Also, $w$ is said to be of class $A_1$ if there is a constant $C_{A_1}$
such that for almost every $x$
\begin{equation*}
M(w)(x)\leq C_{A_1}w(x),
\end{equation*}
where $M(w)$ is the Hardy-Littlewood maximal function of $w$ given by
$$M(w)(x)=\sup_{Q\ni\,x}\frac1{|Q|}\int_Qw(y)dy$$
with cubes $Q$ in $\mathbb{R}^n$.
Then, the following is one of the basic properties of $A_p$ weights ({\it cf. \cite{G}}):
\begin{equation}\label{properties}
A_p\subset A_q,\quad1\leq p<q<\infty,
\end{equation}
and $C_{A_q}$ may be taken to be less than $C_{A_p}$.
It is also known that
\begin{equation}\label{max}
(M(w))^\delta\in A_1,\quad 0<\delta<1,
\end{equation}
with $C_{A_1}$ independent of $w$ if $M(w)(x)<\infty$ for almost every $x$.
(See \cite{St}, V, Subsection 5.2, and also Proposition 2 in \cite{CR}.)
Next, consider the following one-dimensional maximal function of $V$ in the $x_n$ variable:
$$\widetilde{V}(x)=\bigg(\sup_\mu\frac1{2\mu}\int_{x_n-\mu}^{x_n+\mu}|V(x_1,...,x_{n-1},\lambda)|^\beta d\lambda\bigg)^{1/\beta},\quad\beta>1.$$
Namely, $\widetilde{V}(x)=M(|V(x^\prime,\cdot)|^\beta)^{1/\beta}(x_n)$, where $x^\prime=(x_1,...,x_{n-1})$.
If $\beta\leq p$, by H\"older's inequality in $\lambda$ and the fact that
$V\in\mathcal{L}^{2,p}$, it follows that
$$\widetilde{V}(x)\leq C\bigg(\sup_\mu\frac1{2\mu}\int_{x_n-\mu}^{x_n+\mu}|V(x^\prime,\lambda)|^p d\lambda\bigg)^{1/p}<\infty,$$
and so $M(|V(x^\prime,\cdot)|^\beta)(x_n)<\infty$ for almost every $x_n$.
Thus it follows from \eqref{max} that
$$\widetilde{V}=M(|V(x^\prime,\cdot)|^\beta)^{1/\beta}\in A_1(\mathbb{R}),$$
which in turn implies that $\widetilde{V}\in A_2(\mathbb{R})$ (see \eqref{properties})
in the $x_n$ variable uniformly in $x^\prime$.
Finally, if $V\in\mathcal{L}^{2,p}$ for $p>(n-1)/2$, and $\beta<p$, then
\begin{equation}\label{relation}
\widetilde{V}\in\mathcal{L}^{2,p}\quad\text{and}\quad
\|\widetilde{V}\|_{\mathcal{L}^{2,p}}\leq C\|V\|_{\mathcal{L}^{2,p}},
\end{equation}
which can be found in~\cite{CS}. (See Lemma (2.14) there.)

Consequently, if $V\in\mathcal{L}^{2,p}$ for $p>(n-1)/2$, then
so is $\widetilde{V}$.
Also, since $\widetilde{V}\in A_2(\mathbb{R})$ in the $x_n$ variable uniformly in $x^\prime$,
it follows now that
\begin{equation*}
\|u\|_{L_{t,x}^2(\widetilde{V})}\leq C\|\widetilde{V}\|_{\mathcal{L}^{2,p}}
\|(i\partial_t+\Delta+P(D))u\|_{L_{t,x}^2(\widetilde{V}^{-1})}.
\end{equation*}
Since $|V|\leq\widetilde{V}$ and
$\|\widetilde{V}\|_{\mathcal{L}^{2,p}}\leq C\|V\|_{\mathcal{L}^{2,p}}$ (see \eqref{relation}),
this readily implies the desired inequality
\begin{equation*}
\|u\|_{L_{t,x}^2(|V|)}\leq C\|V\|_{\mathcal{L}^{2,p}}
\|(i\partial_t+\Delta+P(D))u\|_{L_{t,x}^2(|V|^{-1})}
\end{equation*}
for the original $V$ without the $A_2$ assumption.


\section{Global unique continuation}\label{sec6}

Given a class $\mathcal{R}$ of resolvent type, from Theorem \ref{thm2},
one can immediately see that the unique continuation
for the Schr\"odinger inequality
\begin{equation}\label{schineq222}
|(i\partial_t+\Delta)u(x,t)|\leq|V(x)u(x,t)|
\end{equation}
holds globally from a half-space in $\mathbb{R}^{n+1}$
under suitable assumptions on the potential $V\in\mathcal{R}$.
This shows a direct link between the unique continuation and
the resolvent estimate, and so a few implications of the resolvent estimates in Section \ref{sec2} are straightforward.
For instance, one can obtain some new results on the unique continuation with $V\in\mathcal{S}_n$.
Recall that $\mathcal{L}^{2,p}\subset\mathcal{S}_n$
for $p>(n-1)/2$, and so $L^{n/2},\,L^{n/2,\infty}\subset\mathcal{S}_n$.
In particular, $\mathcal{K},\,\mathcal{R}\subset\mathcal{S}_3$.
(See the paragraph below Definition \ref{def2.1}.)

Making use of the Carleman estimate in Proposition \ref{prop},
we will prove in this section the following global theorem, Theorem \ref{thm2}.

\begin{thm}[Theorem \ref{thm2}]\label{thm22}
Let $n\geq2$.
Assume that $u\in \mathcal{H}_{t}^{1}(\mathcal{L}_2)\cap \mathcal{H}_{x}^{2}(\mathcal{L}_2)$ is a solution
of \eqref{schineq222} with $V\in\mathcal{R}$ and vanishes in a half space with a unit normal vector $\nu\in\mathbb{R}^{n+1}$.
Then it must be identically zero if $|V|\in A_2(\nu^\prime)$
and $[V]_{\mathcal{R}}<\varepsilon$ for a sufficiently small $\varepsilon>0$.
\end{thm}

First, note from the previous section that the $A_2$ assumption in Proposition \ref{prop}
is superfluous in the cases of $\mathcal{R}=L^{n/2},\,L^{n/2,\infty}$, or $\mathcal{L}^{2,p},\,p>(n-1)/2$.
Hence, as mentioned earlier, the smallness assumption
can be given in these cases by a more local one
\begin{equation}\label{small010}
\sup_{a\in\mathbb{R}}
\lim_{\delta\rightarrow0}[\chi_{S_{a,a+\delta}(\nu^\prime)}V]_{\mathcal{R}}<\varepsilon\quad(\text{if }\nu=(\nu^\prime,0))
\end{equation}
which is trivially satisfied for the case $\mathcal{R}=L^{n/2}$,
and the solution space $\mathcal{L}_2$
can be extended to the whole space $L^2$ in the cases of
$\mathcal{R}=L^{n/2,\infty},\,\mathcal{L}^{2,p},\,(n-1)/2<p<n/2$, or $\mathcal{S}_3$.
In this regard, the Fefferman-Phong class $\mathcal{L}^{2,p}$
is especially good for our theorems,
and it is worth summarizing that point in the following theorem separately.
What we have just remarked above will be clearly demonstrated through the proof of it.

\begin{thm}\label{thm3}
Let $V\in\mathcal{L}^{2,p}$ for $(n-1)/2<p<n/2$, $n\geq3$.
Assume that $u\in\mathcal{H}_{t}^{1}(L^2)\cap\mathcal{H}_{x}^{2}(L^2)$ is a solution of \eqref{schineq222}
and vanishes in a half space with a unit normal vector $\nu\in\mathbb{R}^{n+1}$.
Then it must be identically zero if $\|V\|_{\mathcal{L}^{2,p}}<\varepsilon$ for a sufficiently small $\varepsilon>0$.
If $\nu=(\nu^\prime,0)\in\mathbb{R}^{n+1}$, then the smallness assumption can be replaced by
\begin{equation}\label{small0}
\sup_{a\in\mathbb{R}}
\lim_{\delta\rightarrow0}\|\chi_{S_{a,a+\delta}(\nu^\prime)}V\|_{\mathcal{L}^{2,p}}<\varepsilon
\end{equation}
which is a weaker one.
\end{thm}

The rest of this section is devoted to proving the above-mentioned theorems.

\begin{proof}[Proof of Theorem \ref{thm22}]
By translation we may first assume that
the solution $u$ vanishes in the half space $\{(x,t)\in\mathbb{R}^{n+1}:\langle (x,t),\nu\rangle>0\}$.
Now, from induction it suffices to show that there is
$\sigma>0$ so that $u=0$ in the following strip
\begin{equation}\label{sig}
S_{-\sigma,0}(\nu)=\{(x,t)\in\mathbb{R}^{n+1}:-\sigma<\langle(x,t),\nu\rangle\leq0\}
\end{equation}
with width $\sigma>0$.
We will show this by making use of the Carleman estimate in Proposition \ref{prop}.

First, let $\psi:\mathbb{R}^{n+1}\rightarrow[0,\infty)$ be a
smooth function such that $\text{supp}\,\psi\subset B(0,1)$ and
$$\int_{\mathbb{R}^{n+1}}\psi(x,t)dxdt=1.$$
For $0<\varepsilon<1$, we put
\begin{equation}\label{psi}
\psi_\varepsilon(x,t)=\varepsilon^{-(n+1)}\psi(x/\varepsilon,t/\varepsilon).
\end{equation}
Also, let $\phi:\mathbb{R}^{n+1}\rightarrow[0,1]$ be a smooth
function such that $\phi=1$ in $B(0,1)$ and $\phi=0$ in $\mathbb{R}^{n+1}\setminus B(0,2)$,
and for $R\geq1$ we set $\phi_R(x,t)=\phi(x/R,t/R)$.

Now we apply the Carleman estimate \eqref{Carl} for $V\in\mathcal{R}$ to the following $C_0^\infty$ function
\begin{equation}\label{uu}
\widetilde{u}(x,t)=(u\ast\psi_\varepsilon)(x,t)\phi_R(x,t)
\end{equation}
which is supported in the set
\begin{equation}\label{hh}
H_\varepsilon=\{(x,t)\in\mathbb{R}^{n+1}:\langle(x,t),\nu\rangle\leq\varepsilon\}.
\end{equation}
Then, we see that
$$\big\|\chi_{S_{-\sigma,0}(\nu)}e^{\beta \langle(x,t),\nu\rangle}\widetilde{u}\big\|_{L_{t,x}^2(|V|)}\leq
C[V]_{\mathcal{R}}\big\|\chi_{H_\varepsilon}e^{\beta\langle(x,t),\nu\rangle}(i\partial_t+\Delta)\widetilde{u}\big\|_
{L_{t,x}^2(|V|^{-1})}.$$
By using Fatou's lemma it follows now that
\begin{align}\label{4688}
\nonumber\big\|\chi_{S_{-\sigma,0}(\nu)}e^{\beta \langle(x,t),\nu\rangle}&u\phi_R\big\|_{L_{t,x}^2(|V|)}\\
&\leq C[V]_{\mathcal{R}}\lim_{\varepsilon\rightarrow0}
\big\|\chi_{H_\varepsilon}e^{\beta\langle(x,t),\nu\rangle}(i\partial_t+\Delta)\widetilde{u}\big\|_{L_{t,x}^2(|V|^{-1})},
\end{align}
and note that
$$(i\partial_t+\Delta)\widetilde{u}=((i\partial_t+\Delta)u\ast\psi_\varepsilon)\phi_R
+(u\ast\psi_\varepsilon)(i\partial_t+\Delta)\phi_R
+2\nabla_x(u\ast\psi_\varepsilon)\cdot\nabla_x\phi_R.$$
Hence, since we are assuming that $u\in \mathcal{H}_t^1(\mathcal{L}_2)\cap \mathcal{H}_x^2(\mathcal{L}_2)$,
it follows that
\begin{align}\label{ex}
\nonumber\lim_{\varepsilon\rightarrow0}
\big\|\chi_{H_\varepsilon}e^{\beta\langle(x,t),\nu\rangle}(i\partial_t+\Delta)\widetilde{u}&\big\|_
{L_{t,x}^2(|V|^{-1})}\\
\nonumber&\leq\big\|\chi_{H}e^{\beta\langle(x,t),\nu\rangle}((i\partial_t+\Delta)u)\phi_R\big\|_
{L_{t,x}^2(|V|^{-1})}\\
\nonumber&+\big\|\chi_{H}e^{\beta\langle(x,t),\nu_0\rangle}u(i\partial_t+\Delta)\phi_R\big\|_
{L_{t,x}^2(|V|^{-1})}\\
&+2\big\|\chi_{H}e^{\beta\langle(x,t),\nu_0\rangle}\nabla_xu\cdot\nabla_x\phi_R\big\|_
{L_{t,x}^2(|V|^{-1})},
\end{align}
where
\begin{equation}\label{h}
H=\{(x,t)\in\mathbb{R}^{n+1}:\langle(x,t),\nu\rangle\leq0\}.
\end{equation}
Then, by letting $R\rightarrow\infty$ we see that
\begin{align*}
\lim_{R\rightarrow\infty}\lim_{\varepsilon\rightarrow0}
\big\|\chi_{H_\varepsilon}e^{\beta\langle(x,t),\nu\rangle}(i\partial_t+\Delta)\widetilde{u}\big\|_{L_{t,x}^2(|V|^{-1})}
\leq\big\|\chi_{H}e^{\beta\langle(x,t),\nu\rangle}(i\partial_t+\Delta)u\big\|_{L_{t,x}^2(|V|^{-1})}.
\end{align*}
By applying Fatou's lemma to \eqref{4688} and using this, we conclude that
\begin{equation}\label{456}
\big\|\chi_{S_{-\sigma,0}(\nu)}e^{\beta \langle(x,t),\nu\rangle}u\big\|_{L_{t,x}^2(|V|)}
\leq C[V]_{\mathcal{R}}\big\|\chi_{H}e^{\beta\langle(x,t),\nu\rangle}(i\partial_t+\Delta)u\big\|_
{L_{t,x}^2(|V|^{-1})}.
\end{equation}

Now we decompose the norm in the right-hand side of \eqref{456} into two parts
$$\big\|\chi_{S_{-\sigma,0}(\nu)}e^{\beta\langle(x,t),\nu\rangle}(i\partial_t+\Delta)u\big\|_{L_{t,x}^2(|V|^{-1})}$$
and
$$\big\|\chi_{H\setminus S_{-\sigma,0}(\nu)}e^{\beta\langle(x,t),\nu\rangle}(i\partial_t+\Delta)u\big\|_
{L_{t,x}^2(|V|^{-1})}.$$
Then, since $u$ is a solution of the Schr\"odinger inequality and we are assuming that
$[V]_{\mathcal{R}}<\varepsilon$,
the first part can be absorbed into the left-hand side of \eqref{456} in the following way:
\begin{align*}
\big\|\chi_{S_{-\sigma,0}(\nu)}e^{\beta\langle(x,t),\nu\rangle}(i\partial_t+\Delta)u\big\|_{L_{t,x}^2(|V|^{-1})}
&\leq\big\|\chi_{S_{-\sigma,0}(\nu)}e^{\beta\langle(x,t),\nu\rangle}Vu\big\|_{L_{t,x}^2(|V|^{-1})}\\
&=\big\|\chi_{S_{-\sigma,0}(\nu)}e^{\beta\langle(x,t),\nu\rangle}u\big\|_{L_{t,x}^2(|V|)}.
\end{align*}
On the other hand, the second part is bounded for $\beta>0$ by
$$e^{-\beta\sigma}\big\|\chi_{H\setminus S_{-\sigma,0}(\nu)}(i\partial_t+\Delta)u\big\|_
{L_{t,x}^2(|V|^{-1})}$$
because $\langle(x,t),\nu\rangle\leq-\sigma$ in the set $H\setminus S_{-\sigma,0}(\nu)$.
Consequently, we get
\begin{align}\label{am}
\big\|\chi_{S_{-\sigma,0}(\nu)}e^{\beta (\langle(x,t),\nu\rangle+\sigma)}u\big\|_{L_{t,x}^2(|V|)}
&\leq C\|\chi_{H\setminus S_{-\sigma,0}(\nu)}(i\partial_t+\Delta)u\|_{L_{t,x}^2(|V|^{-1})}\\
\nonumber&<\infty.
\end{align}
Finally, by letting $\beta\rightarrow\infty$ it follows that $u=0$ in the strip $S_{-\sigma,0}(\nu)$.
This completes the proof.
\end{proof}

\begin{proof}[Proof of Theorem \ref{thm3}]
To prove the first assertion of the theorem, we only need to show that
the solution space $\mathcal{L}_2$ can be extended to the whole space $L^2$.
For this we will use the fact that $1/|x|^2\in\mathcal{L}^{2,p}$
and the Fefferman-Phong inequality \eqref{ineq}.
First, note that
\begin{equation}\label{note2}
\|(f\ast g)\chi_{B(0,2R)}\|_{L_{t,x}^2(|V|^{-1})}\leq CR^2\|f\|_{L_{t,x}^2}\|g\|_{L_{t,x}^1}
\end{equation}
if $|V|\geq C/|x|^2$ for $|x|\leq2R$. Since we can choose $\delta>0$ small enough so that $\|\delta/|x|^2\|_{\mathcal{L}^{2,p}}=\delta\|1/|x|^2\|_{\mathcal{L}^{2,p}}$ is sufficiently small,
and
$$|(i\partial_t+\Delta)u|\leq|Vu|\leq|(|V|+\delta/|x|^2)u|,$$
we may assume that $|V|\geq\delta/|x|^2$ by replacing $V$ with $|V|+\delta/|x|^2$.
Then, by using \eqref{note2}, one can get \eqref{ex} assuming
$u\in \mathcal{H}_t^1(L^2)\cap \mathcal{H}_x^2(L^2)$
instead of
$u\in \mathcal{H}_t^1(\mathcal{L}_2)\cap \mathcal{H}_x^2(\mathcal{L}_2)$.
Next, one can use the Fefferman-Phong inequality \eqref{ineq} to bound the right-hand side of \eqref{am}.
Indeed, by applying the inequality, it follows that
\begin{align*}
\|\chi_{H\setminus S_{-\sigma,0}(\nu)}(i\partial_t+\Delta)u\|_{L_{t,x}^2(|V|^{-1})}
&\leq\|Vu\|_{L_{t,x}^2(|V|^{-1})}\\
&=\|u\|_{L_{t,x}^2(|V|)}\\
&\leq C\|V\|_{\mathcal{L}^{2,p}}\|\nabla u\|_{L_{t,x}^2}<\infty.
\end{align*}
Thus, the previous proof entirely works for $u\in\mathcal{H}_t^1(L^2)\cap \mathcal{H}_x^2(L^2)$.

\smallskip

For the second assertion of the theorem,
we consider the following $C_0^\infty$ function instead of \eqref{uu}:
$$\widetilde{u}_\sigma(x,t)=(u\ast\psi_\varepsilon)(x,t)\phi_R(x,t)\varphi_\sigma(\langle(x,t),\nu\rangle),$$
where $\varphi_\sigma(r)=\varphi(r/\sigma)$ for $\varphi:\mathbb{R}\rightarrow[0,1]$ which is a smooth
function equal to $1$ in $\{-1/2<r\leq0\}$ and equal to $0$
in $\{r<-1 \text{ or } r\geq1/2\}$.
Note first that if $\nu=(\nu^\prime,0)$
\begin{equation}\label{note22}
S_{a,a+\delta}(\nu)=\mathbb{R}\times S_{a,a+\delta}(\nu^\prime),
\end{equation}
and so
$$\big\|\chi_{S_{a,a+\delta}(\nu)}(x,t)e^{\beta \langle(x,t),\nu\rangle}\widetilde{u}_\sigma\big\|_{L_{t,x}^2(|V|)}
=\big\|e^{\beta \langle(x,t),\nu\rangle}\widetilde{u}_\sigma\big\|_
{L_{t,x}^2(|\chi_{S_{a,a+\delta}(\nu^\prime)}(x)V|)}.$$
Then, since $\widetilde{u}_\sigma$ is supported in the strip $S_{-\sigma,\varepsilon}(\nu)$,
by applying the Carleman estimate \eqref{Carl} with $V=\chi_{S_{-3\sigma/2,\sigma/2}(\nu^\prime)}V$ to $\widetilde{u}_\sigma$, one can see that
\begin{align*}
\big\|\chi_{S_{-\sigma/2,0}(\nu)}&e^{\beta \langle(x,t),\nu\rangle}\widetilde{u}_\sigma\big\|_{L_{t,x}^2(|V|)}\\
&\leq C\|\chi_{S_{-3\sigma/2,\sigma/2}(\nu^\prime)}V\|_{\mathcal{L}^{2,p}}
\big\|\chi_{S_{-\sigma,\varepsilon}(\nu)}e^{\beta\langle(x,t),\nu\rangle}(i\partial_t+\Delta)\widetilde{u}_\sigma\big\|_
{L_{t,x}^2(|V|^{-1})}
\end{align*}
because $S_{-\sigma,\varepsilon}(\nu^\prime)\subset S_{-3\sigma/2,\sigma/2}(\nu^\prime)$ for $\varepsilon<\sigma/2$.
Hence, by the same limiting argument as before, it follows that
\begin{align*}
\big\|\chi_{S_{-\sigma/2,0}(\nu)}&e^{\beta \langle(x,t),\nu\rangle}u\big\|_{L_{t,x}^2(|V|)}\\
&\leq C\|\chi_{S_{-3\sigma/2,\sigma/2}(\nu^\prime)}V\|_{\mathcal{L}^{2,p}}
\big\|\chi_{S_{-\sigma,0}(\nu)}e^{\beta\langle(x,t),\nu\rangle}(i\partial_t+\Delta)u\big\|_{L_{t,x}^2(|V|^{-1})}.
\end{align*}
Now, from the smallness assumption \eqref{small0} we can choose $\sigma>0$ small enough so that
$$C\|\chi_{S_{-3\sigma/2,\sigma/2}(\nu^\prime)}V\|_{\mathcal{L}^{2,p}}<1/2.$$
By decomposing $\chi_{S_{-\sigma,0}(\nu)}=\chi_{S_{-\sigma/2,0}(\nu)}+\chi_{S_{-\sigma,-\sigma/2}(\nu)}$
and repeating the previous argument, this leads to
\begin{align*}
\big\|\chi_{S_{-\sigma/2,0}(\nu)}e^{\beta (\langle(x,t),\nu\rangle+\sigma/2)}u\big\|_{L_{t,x}^2(|V|)}
&\leq C\|\chi_{S_{-\sigma,-\sigma/2}(\nu)}(i\partial_t+\Delta)u\|_{L_{t,x}^2(|V|^{-1})}\\
&< \infty.
\end{align*}
Hence, by letting $\beta\rightarrow\infty$ it follows that $u=0$ in the strip $S_{-\sigma/2,0}(\nu)$.
This completes the proof.
\end{proof}


\section{Local unique continuation}\label{sec7}

Now we turn to the local unique continuation theorem.
Recall that there exists a smooth potential $V$
such that $(i\partial_t+\Delta)u=V(x,t)u$,
$0\in\text{supp}\,u$, and $u=0$ on $\{(x,t)\in\mathbb{R}^{n+1}:x_1<0\}$ in a neighborhood of the origin.
See the paragraph above Theorem \ref{thm6}.
Since $0\in\text{supp}\,u$, the solution $u$ cannot vanish near the origin
across the hyperplane $\{(x,t)\in\mathbb{R}^{n+1}:x_1=0\}$.
This shows that the Schr\"odinger equation does not have, as a rule, a property of unique continuation
locally across a hyperplane in $\mathbb{R}^{n+1}$.
However, our result in Theorem \ref{thm6} says that the unique continuation for the Schr\"odinger inequality
\begin{equation}\label{schineq223}
|(i\partial_t+\Delta)u(x,t)|\leq|V(x)u(x,t)|
\end{equation}
can hold locally across a hypersurface on a sphere in $\mathbb{R}^{n+1}$
into an interior region of the sphere.

\begin{thm}[Theorem \ref{thm6}]\label{thm66}
Let $n\geq2$ and let $S_r^n$ be a sphere in $\mathbb{R}^{n+1}$ with radius $r$.
Assume that $u\in\mathcal{H}_{t}^{1}(\mathcal{L}_2)\cap \mathcal{H}_{x}^{2}(\mathcal{L}_2)$
is a solution of \eqref{schineq223} with $V\in\mathcal{R}$
and vanishes on an exterior neighborhood of $S_r^n$ in a neighborhood of a point $p\in S_r^n$.
Let $\nu$ be the unit outward normal vector of $S_r^n$ at $p$.
Then it follows that $u\equiv0$ in a neighborhood of $p$,
if $|V|\in A_2(\nu^\prime)$ and $[V]_{\mathcal{R}}<\varepsilon$ for a sufficiently small $\varepsilon>0$.
\end{thm}

In the same cases as in the global theorem,
the solution space $\mathcal{L}_2$ can be extended to the whole space $L^2$
and the smallness assumption can be given by a more local one
$$\lim_{\delta\rightarrow0}[\chi_{B_\delta(p^\prime)}V]_{\mathcal{R}}<\varepsilon,$$
where $B_\delta(p)$ denotes a ball centered at $p$ with radius $\delta$.
Unlike \eqref{small010}, we do not need to assume $\nu=(\nu^\prime,0)$,
since the property \eqref{note22} is trivially satisfied for balls in the sense that
$x\in B_\delta(p^\prime)$ whenever $(x,t)\in B_\delta(p)$.

\begin{proof}[Proof of Theorem \ref{thm66}]
For simplicity of notation, we shall first assume that the point $p$ is the origin.
Then, we are assuming that the solution $u$ vanishes on an exterior neighborhood of a sphere $S_r^n$
in a neighborhood $\mathcal{N}$ of the origin,
and $\nu$ is the unit outward normal vector of $S_r^n$ at the origin.
Hence, there is $\delta>0$ so that $B_\delta(0)\subset\mathcal{N}$,
and since $u$ vanishes on an exterior neighborhood of $S_r^n$ in $\mathcal{N}$,
we can choose $\sigma>0$ small enough so that $u$ vanishes on the set
$$(B_\delta(0)\setminus B_{\delta/2}(0))\cap S_{-\sigma,0}(\nu),$$
where $S_{-\sigma,0}(\nu)$ is the strip given in \eqref{sig}.
Of course, the solution is also vanishing in the set
$$B_\delta(0)\cap\{(x,t)\in\mathbb{R}^{n+1}:\langle(x,t),\nu\rangle>0\}.$$
Now it suffices to show that $u$ vanishes in the set $B_{\delta/2}(0)\cap S_{-\sigma,0}(\nu)$.

To show this, we apply the Carleman estimate in Proposition \ref{prop}
to the $C_0^\infty$ function
$$\widetilde{u}(x,t)=(u\ast\psi_\varepsilon)(x,t)\eta(x,t),$$
where $\psi_\varepsilon$ is given in \eqref{psi} and $\eta:\mathbb{R}^{n+1}\rightarrow[0,1]$ is a smooth
function such that $\eta=1$ in $B_{\delta/2}(0)$ and $\eta=0$
in $\mathbb{R}^{n+1}\setminus B_\delta(0)$.
In fact, since $\widetilde{u}$ is supported in the set
$B_\delta(0)\cap H_\varepsilon$, where $H_\varepsilon$ is given in \eqref{hh},
one can see that
\begin{align*}
\big\|\chi_{B_{\delta/2}(0)\cap S_{-\sigma,0}(\nu)}e^{\beta \langle(x,t),\nu\rangle}&\widetilde{u}\big\|_{L_{t,x}^2(|V|)}\\
&\leq C[V]_{\mathcal{R}}\big\|\chi_{B_\delta(0)\cap H_\varepsilon}e^{\beta\langle(x,t),\nu\rangle}(i\partial_t+\Delta)\widetilde{u}\big\|_
{L_{t,x}^2(|V|^{-1})}.
\end{align*}
Then, by the same limiting argument as before, it follows that
\begin{align}\label{seoi}
\nonumber\big\|\chi_{B_{\delta/2}(0)\cap S_{-\sigma,0}(\nu)}
&e^{\beta \langle(x,t),\nu\rangle}u\big\|_{L_{t,x}^2(|V|)}\\
&\leq C[V]_{\mathcal{R}}
\big\|\chi_{B_\delta(0)\cap H}e^{\beta\langle(x,t),\nu\rangle}(i\partial_t+\Delta)(u\eta)\big\|_
{L_{t,x}^2(|V|^{-1})},
\end{align}
where $H$ is given in \eqref{h}.

Now, we decompose the norm in the right-hand side of \eqref{seoi} into two parts.
Then the first part
$$\big\|\chi_{B_{\delta/2}(0)\cap S_{-\sigma,0}(\nu)}e^{\beta\langle(x,t),\nu\rangle}(i\partial_t+\Delta)u\big\|_{L_{t,x}^2(|V|^{-1})}$$
can be absorbed into the left-hand side of \eqref{seoi} as before,
while the second part
$$\big\|\chi_{(B_\delta(0)\cap H)\setminus(B_{\delta/2}(0)\cap S_{-\sigma,0}(\nu))}e^{\beta\langle(x,t),\nu\rangle}(i\partial_t+\Delta)(u\eta)\big\|_
{L_{t,x}^2(|V|^{-1})}$$
is bounded for $\beta>0$ by
$$e^{-\beta\sigma}\big\|\chi_{(B_\delta(0)\cap H)\setminus(B_{\delta/2}(0)\cap S_{-\sigma,0}(\nu))}(i\partial_t+\Delta)(u\eta)\big\|_{L_{t,x}^2(|V|^{-1})},$$
since $u$ vanishes on the set $(B_\delta(0)\setminus B_{\delta/2}(0))\cap S_{-\sigma,0}(\nu)$
and $\langle(x,t),\nu\rangle\leq-\sigma$ in the set $H\setminus S_{-\sigma,0}(\nu)$.
At this point, it is worth noting that the assumption that
$u$ vanishes on an exterior neighborhood of a sphere is crucial in this second part
in order to guarantee that $u$ vanishes on the set $(B_\delta(0)\setminus B_{\delta/2}(0))\cap S_{-\sigma,0}(\nu)$.
Consequently, we see that
\begin{align*}
\big\|\chi_{B_{\delta/2}(0)\cap S_{-\sigma,0}(\nu)}e^{\beta (\langle(x,t),\nu\rangle+\sigma)}u\big\|_{L_{t,x}^2(|V|)}
&\leq C\big\|(i\partial_t+\Delta)(u\eta)\big\|_{L_{t,x}^2(|V|^{-1})}\\
&<\infty
\end{align*}
since $\eta\in C_0^\infty$ and $u\in \mathcal{H}_t^1(\mathcal{L}_2)\cap \mathcal{H}_x^2(\mathcal{L}_2)$.
Finally, by letting $\beta\rightarrow\infty$
it follows that $u=0$ in the set $B_{\delta/2}(0)\cap S_{-\sigma,0}(\nu)$, as desired.
This completes the proof.
\end{proof}



\section{Further applications}\label{sec8}

In this section we present a few applications of our resolvent estimates to global well-posedness
of the Schr\"odinger equation in weighted $L^2$ spaces.

Let us first consider the following Cauchy problem with an initial datum $f$ and a forcing term $F$:
\begin{equation*}
\left\{
\begin{array}{ll}
i\partial_tu+\Delta u=F(x,t),\\
u(x,0)=f(x).
\end{array}\right.
\end{equation*}
By Duhamel's principle, the solution is then given by
\begin{equation*}
u(x,t)=e^{it\Delta}f(x)-i\int_0^te^{i(t-s)\Delta}F(\cdot,s)ds,
\end{equation*}
where the first and second terms correspond to the solutions of the homogeneous ($F=0$)
and inhomogeneous ($f=0$) problems, respectively.
Now we have the following Strichartz estimates for the solutions in weighted $L^2$ spaces:

\begin{prop}\label{str}
Let $n\geq2$ and let $V:\mathbb{R}^n\rightarrow\mathbb{C}$ be such that the resolvent estimate
\begin{equation*}
\|R_0(z)f\|_{L^2(|V|)}\leq C(V)\|f\|_{L^2(|V|^{-1})}
\end{equation*}
holds uniformly in $z\in\mathbb{C}\setminus\mathbb{R}$.
Then the following estimates hold:
\begin{equation}\label{homo8}
\big\|e^{it\Delta}f\big\|_{L_{t,x}^2(|V|)}\leq CC(V)^{1/2}\|f\|_2,
\end{equation}
\begin{equation}\label{dual8}
\bigg\|\int_{-\infty}^\infty e^{-is\Delta}F(\cdot,s)ds\bigg\|_2
\leq CC(V)^{1/2}\|F\|_{L_{t,x}^2(|V|^{-1})},
\end{equation}
\begin{equation}\label{inho8}
\bigg\|\int_0^t e^{i(t-s)\Delta}F(\cdot,s)ds\bigg\|_{L_{t,x}^2(|V|)}
\leq CC(V)\|F\|_{L_{t,x}^2(|V|^{-1})}.
\end{equation}
\end{prop}

\begin{rem}
Given a class $\mathcal{R}$ of resolvent type, this proposition clearly holds for $V\in\mathcal{R}$ with $C(V)\sim[V]_{\mathcal{R}}$.
In the cases of $\mathcal{R}=\mathcal{L}^{2,p},\,p>(n-1)/2$, and $\mathcal{S}_3$,
the above Strichartz estimates can be found in \cite{RV2} and \cite{BBRV}, respectively.
From the resolvent estimates in Theorem \ref{thm2.3}, these previous results are extended to the class $\mathcal{S}_n$.
\end{rem}

\begin{proof}[Proof of Proposition \ref{str}]
We have already proved the homogeneous estimate \eqref{homo8} in Section \ref{sec4} in two ways,
and by duality it is equivalent to \eqref{dual8}.
For the inhomogeneous estimate \eqref{inho8}, following the simple argument used for Proposition 2.5 in \cite{RV},
we may write
\begin{equation}\label{dec}
\int_0^t e^{i(t-s)\Delta}F(\cdot,s)ds=v(x,t)-e^{it\Delta}\bigg(\int_{-\infty}^\infty e^{-is\Delta}[sgn(s)F(\cdot,s)]ds\bigg)(x),
\end{equation}
where
\begin{equation*}
v(x,t)=\lim_{\varepsilon\rightarrow0}
\mathcal{F}^{-1}\bigg(\frac{\widehat{F}(\xi,\tau)}{-|\xi|^2-\tau+i\varepsilon}\bigg)(x,t).
\end{equation*}
Combining \eqref{homo8} and \eqref{dual8},
we get the desired estimate for the second term in the right-hand side of \eqref{dec}.
To bound the first term, all we have to do is just to use the resolvent estimate in the same way as in \eqref{multi2}.
Now \eqref{inho8} is proved. But here, we point out that it can be also obtained more directly
appealing to Kato $H$-smoothing theory as in the homogeneous case.
In fact the following\,\footnote{\,For the proof, see, for example, \cite{M}.} is due to Kato \cite{K}:
If \eqref{ksm} in Lemma \ref{lem4.4} holds, then
 $$\int_{\mathbb{R}}\bigg\|\int_0^tTe^{i(t-s)H}T^\ast F(\cdot,s)ds\bigg\|_{\mathcal{H}}^2dt\leq \widetilde{C}^2\int_{\mathbb{R}}\|F(\cdot,t)\|_{\mathcal{H}}^2dt.$$
By applying this, with $H=-\Delta$, $\mathcal{H}=L^2$ and $T:f\mapsto|V|^{1/2}f$ as before,
we can get alternatively \eqref{inho8}.
\end{proof}

Making use of Proposition \ref{str}, we now obtain the global well-posedness of the following Cauchy problem for the Schr\"odinger equation:
\begin{equation}\label{ini}
\left\{
\begin{array}{ll}
i\partial_tu+\Delta u-V(x)u=F(x,t),\\
u(x,0)=u_0(x).
\end{array}\right.
\end{equation}

\begin{thm}\label{well}
Let $n\geq2$.
Given a class $\mathcal{R}$ of resolvent type, we assume that $V\in\mathcal{R}$ and $[V]_{\mathcal{R}}$ is sufficiently small.
If $u_0\in L^2$ and $F\in L_{t,x}^2(|V|^{-1})$, then there exists a unique solution of \eqref{ini} in the weighted space $L_{t,x}^2(|V|)$.
Furthermore, the solution $u$ belongs to $C_tL_x^2$, and satisfies
\begin{equation}\label{1-1}
\|u\|_{L^2(|V|)}\leq C[V]_{\mathcal{R}}^{1/2}\|u_0\|_{L^2}+C[V]_{\mathcal{R}}\|F\|_{L_{t,x}^2(|V|^{-1})}
\end{equation}
and
\begin{equation}\label{1-2}
\sup_{t\in\mathbb{R}}\|u\|_{L_x^2}\leq C\|u_0\|_{L^2}+C[V]_{\mathcal{R}}^{1/2}\|F\|_{L_{t,x}^2(|V|^{-1})}.
\end{equation}
\end{thm}

\begin{proof}
The proof is quite standard once one has the weighted $L^2$ Strichartz estimates
in Proposition \ref{str}.

Let us first consider the following integral formulation of \eqref{ini}:
\begin{equation}\label{sol}
u(x,t)=e^{it\Delta}u_0(x)-i\int_0^te^{i(t-s)\Delta}F(\cdot,s)ds+\Phi(u)(x,t),
\end{equation}
where
$$\Phi(u)(x,t)=-i\int_0^t e^{i(t-s)\Delta}(Vu)(\cdot,s)ds.$$
Then, since $u_0\in L^2$ and $F\in L^2(|V|^{-1})$,
from Proposition \ref{str} and
$$(I-\Phi)(u)=e^{it\Delta}u_0(x)-i\int_0^t e^{i(t-s)\Delta}F(\cdot,s)ds,$$
where $I$ is the identity operator, it follows that
$$(I-\Phi)(u)\in L^2(|V|).$$
Hence, the global existence in the theorem follows if
the operator $I-\Phi$ has an inverse in the space $L^2(|V|)$.
It is well known that this holds if the operator norm of $\Phi$ in $L^2(|V|)$ is strictly less than $1$.
Namely, it is enough to show that
$$\|\Phi(u)\|_{L^2(|V|)}<\|u\|_{L^2(|V|)}.$$
But, using the inhomogeneous estimate \eqref{inho8}, we can get this as follows:
\begin{align}\label{inv}
\nonumber\|\Phi(u)\|_{L^2(|V|)}&\leq C[V]_{\mathcal{R}}\|Vu\|_{L^2(|V|^{-1})}\\
\nonumber&=C[V]_{\mathcal{R}}\|u\|_{L^2(|V|)}\\
&<\frac12\|u\|_{L^2(|V|)}.
\end{align}
Here, for the last inequality we used the smallness assumption on the quantity $[V]_{\mathcal{R}}$.

On the other hand, by applying Proposition \ref{str} and \eqref{inv} to \eqref{sol}, one can easily see that
\begin{align}\label{1123}
\nonumber\|u\|_{L^2(|V|)}&\leq C\big\|e^{it\Delta}u_0 \big\|_{L^2(|V|)}
+C\bigg\|\int_{0}^{t}e^{i(t-s)\Delta}F(\cdot,s)ds\bigg\|_{L^2(|V|)}\\
&\leq C[V]_{\mathcal{R}}^{1/2}\|u_0\|_{L^2}+C[V]_{\mathcal{R}}\|F\|_{L_{t,x}^2(|V|^{-1})}.
\end{align}
Hence \eqref{1-1} is proved.
To show \eqref{1-2}, we will use \eqref{1123} and the estimate \eqref{dual8}.
First, from \eqref{sol}, \eqref{dual8}, and the simple fact that $e^{it\Delta}$ is an isometry in $L^2$,
one can see that
$$\|u\|_{L_x^2}\leq C\|u_0\|_{L^2}
+C[V]_{\mathcal{R}}^{1/2}\|F\|_{L^2(|V|^{-1})}
+C[V]_{\mathcal{R}}^{1/2}\|Vu\|_{L^2(|V|^{-1})}.$$
Since $\|Vu\|_{L^2(|V|^{-1})}=\|u\|_{L^2(|V|)}$ and $[V]_{\mathcal{R}}$ is small,
by combining the above inequality and \eqref{1123} we get
$$\|u\|_{L_x^2}\leq C\|u_0\|_{L^2}+C[V]_{\mathcal{R}}^{1/2}\|F\|_{L_{t,x}^2(|V|^{-1})}$$
as desired.
Finally, it is an elementary matter to check that $u\in C_tL_x^2$.
Now the proof is completed.
\end{proof}

\section{Concluding remarks}\label{sec9}

The abstract Carleman estimate \eqref{Carl} in Proposition \ref{prop}
can be modified to work for the case of time-dependent potentials $V(x,t)$ such that
\begin{equation}\label{w}
\sup_{t\in\mathbb{R}}|V(x,t)|\leq W(x)\in\mathcal{R}.
\end{equation}
Taking the sup in $t$ has been sometimes used for time-dependent potentials
in other problems (\cite{RV,RV2,BRV,BBCRV,BBRV}) concerning Schr\"odinger equations.
Since $|V(x,t)|\leq W(x)$ for almost every $t$,
the following corollary is an immediate consequence of Proposition \ref{prop}.

\begin{cor}
Let $n\geq2$ and let $V:\mathbb{R}^{n+1}\rightarrow\mathbb{C}$ be such that \eqref{w} holds.
Then we have for $u\in C_0^\infty(\mathbb{R}^{n+1})$
\begin{equation*}
\big\|e^{\beta\langle(x,t),\nu\rangle}u\big\|_{L_{t,x}^2(|V(x,t)|)}\leq
C[W]\big\|e^{\beta\langle(x,t),\nu\rangle}(i\partial_t+\Delta)u\big\|_{L_{t,x}^2(|V(x,t)|^{-1})}
\end{equation*}
with a constant $C$ independent of $\beta\in\mathbb{R}$ and $\nu\in\mathbb{R}^{n+1}$,
if $W\in A_2(\nu^\prime)$.
\end{cor}

This is naturally expected to lead to unique continuation for the time-dependent potential case
\begin{equation}\label{tidep}
|(i\partial_t+\Delta)u(x,t)|\leq|V(x,t)u(x,t)|
\end{equation}
where $V$ satisfies \eqref{w}.
In fact the same type of argument used for time-independent potentials
works clearly for this case if the assumptions on the potential in Theorems \ref{thm2} and \ref{thm6} are given for $W(x)$ instead of $V(x)$.
As a consequence, one can obtain some new results on the unique continuation for \eqref{tidep}.
We omit the details.

It is also straightforward that the well-posedness result in the previous section can be applied to the case of time-dependent potentials $V(x,t)$
satisfying \eqref{w}.
This case has been studied in \cite{RV2} and \cite{BBRV}
for $\mathcal{R}=\mathcal{L}^{2,p},\,p>(n-1)/2$, and $\mathcal{S}_3$, respectively.
These previous results are now extended to the class $\mathcal{S}_n$
as an immediate consequence of our resolvent estimates.


\end{document}